\documentclass[12pt,a4paper]{amsart}
\usepackage{amsfonts,amsmath,amssymb}
\usepackage{hyperref}
\usepackage{latexsym,fullpage,amsfonts,amssymb,amsmath,amscd,graphics,epic}
\usepackage[all]{xy}
\usepackage{amssymb,amsbsy,amsthm,amsxtra}
\usepackage[usenames]{color}
\usepackage{amscd}
\usepackage{amsthm}
\usepackage{amsfonts}
\usepackage{amssymb}
\usepackage{mathrsfs}
\usepackage{url}
\usepackage{bbm}
\usepackage{wasysym}
\usepackage{enumitem}
\usepackage{framed}

\theoremstyle{plain}
\newtheorem*{theorem*}{Theorem}
\newtheorem*{remark*}{Remark}
\newtheorem*{example*}{Example}
\newtheorem{lemma}{Lemma}[subsection]
\newtheorem{proposition}[lemma]{Proposition}
\newtheorem{corollary}[lemma]{Corollary}
\newtheorem{theorem}[lemma]{Theorem}

\newtheorem*{conjecture*}{Conjecture}

\newtheorem{sublemma}[lemma]{Sublemma}

\newtheorem{introtheorem}{Theorem}

\theoremstyle{definition}
\newtheorem{definition}[lemma]{Definition}

\newtheorem{example}[lemma]{Example}

\theoremstyle{remark}
\newtheorem{remark}[lemma]{Remark}

\newtheorem{notation}[lemma]{Notation}

\newcommand{\Hom}{\operatorname{Hom}}

\newcommand{\triv}{{\mathbbm 1}}

\newcommand{\id}{\operatorname{Id}}

\renewcommand{\Im}{\operatorname{Im}}
\newcommand{\Ker}{\operatorname{Ker}}

\newcommand{\F}{{\mathcal F}}

\newcommand{\rk}{{\operatorname{rk}}}
\newcommand{\sdim}{{\operatorname{sdim}}}

\newcommand{\C}{{\mathbb C}}
\newcommand{\Z}{{\mathbb Z}}

\newcommand{\eps}{{\varepsilon}}

\newcommand{\Lam}{{\Lambda}}
\newcommand{\lam}{{\lambda}}

\newcommand{\gl}{\mathfrak{gl}}
\newcommand{\abs}[1]{\left|{#1}\right|}

\newcommand{\p}{\mathfrak{p}}

\newcommand{\Rep}{\mathrm{Rep}}

\newcommand{\fb}{\mathfrak{b}}

\newcommand{\InnaA}[1]{{#1}}
\newcommand{\InnaB}[1]{{#1}}
\newcommand{\InnaC}[1]{{#1}}
\newcommand{\InnaD}[1]{{#1}}
\newcommand{\VeraA}[1]{{#1}}
\newcommand{\VeraB}[1]{{#1}}

\newcommand{\comment}[1]{}

\def\quotient#1#2{%
    \raise1ex\hbox{$#1$}\Big/\lower1ex\hbox{$#2$}%
}

\begin{document}

\date{\today}
\title{Duflo-Serganova functor and superdimension formula for the periplectic Lie superalgebra}
 \author{Inna Entova-Aizenbud, Vera Serganova}
\address{Inna Entova-Aizenbud, Dept. of Mathematics, Ben Gurion University,
Beer-Sheva,
Israel.}
\email{entova@bgu.ac.il}
\address{Vera Serganova, Dept. of Mathematics, University of California at
Berkeley,
Berkeley, CA 94720.}
\email{serganov@math.berkeley.edu}

\dedicatory{To Pavel Etingof for his 50th birthday}
%
%
\begin{abstract}
In this paper, we study the representations of the periplectic Lie superalgebra using the Duflo-Serganova functor. Given a simple $\p(n)$-module $L$ and a certain element $x\in \p(n)$ of rank $1$, we give an explicit description of the composition factors of the $\p(n-1)$-module $DS_x(L)$, which is defined as the homology of the complex $$\Pi M \xrightarrow{x} M \xrightarrow{x} \Pi M.$$

In particular, we show that this $\p(n-1)$-module is multiplicity-free.
 
 We then use this result to give a simple explicit combinatorial formula for the superdimension of a simple integrable finite-dimensional $\p(n)$-module, based on its highest weight. 
 
 In particular, this reproves the Kac-Wakimoto conjecture for $\p(n)$, which was proved earlier by the authors in \cite{ES_KW}.
\end{abstract}

\maketitle
\setcounter{tocdepth}{3}

\section{Introduction}
\subsection{}
Consider a complex vector superspace $V$, and let $\C^{0|1}$ be the odd one-dimensional vector superspace. 

The (complex) periplectic Lie superalgebra $\p(V)$ is the Lie superalgebra of 
endomorphisms of a complex vector superspace $V$ \VeraA{preserving} a non-degenerate 
symmetric form $\omega: \VeraA{S^2 V} \to \C^{0|1}$ (this form is also referred to as an 
``odd form''). When $V = \C^{n|n}$ and $\omega_n: \C^{n|n} \otimes \C^{n|n} \to \C^{0|1}$ pairs the even and odd 
parts of this vector superspace, we denote this Lie superalgebra by $\p(n):=\p(V)$.

The periplectic Lie superalgebra $\p(n)$
\VeraA{has} an interesting non-semisimple representation theory; some results 
on the category of finite-dimensional integrable representations 
of $\p(n)$ can be found in \cite{BDE:periplectic, Chen:periplectic, Coulembier:Brauer, DLZ:fft_ortho, Gor:center, HIR19, IRS19, Moon, Ser:periplectic}.

An important tool in studying representations of Lie superalgebras, particularly the connection between representation theory of Lie superalgebras of same type but different rank, is the Duflo-Serganova functor. Given an odd element $x\in \p(n)$ satisfying $[x,x]=0$, and a $\p(n)$-module $M$, we denote by $DS_x M$ the homology of the complex $$\Pi M \xrightarrow{x} M \xrightarrow{x} \Pi M.$$ The resulting homology is a module over the Lie superalgebra $DS_x(\p(n))$, and $DS_x$ can be seen as a symmetric monoidal functor $$\F_n \to \Rep(DS_x(\p(n))).$$ This functor is called the Duflo-Serganova functor. 

This functor has been introduced in \cite{DufloSer:functor} in a general Lie superalgebra setting. The Duflo-Serganova functor has been studied extensively for different Lie superalgebras, see for example \cite{ES_perip_Deligne, ES_KW, GorSerDS, HeiWei14, HR18, IRS19, Ser:sdim}. Its precise effect in the periplectic case has been unknown until now, although it was shown that it can be used to compute Grothendieck rings for $\p(n)$, see \cite{IRS19}. 

For a suitably chosen $x\in \p(n)$ of rank $1$, the Lie superalgebra $DS_x(\p(n))$ is isomorphic to $\p(n-1)$. In that case, the Duflo-Serganova functor becomes simply a symmetric monoidal functor $DS_x: \F_n \to \F_{n-1}$, \VeraA{where by $\F_n$ we denote the category of finite-dimensional $\p(n)$-modules \InnaC{where the $\p(n)_{\bar{0}} \cong \gl_n$ action} can be lifted to an action of $GL(n)$.}

Although this functor is not exact on either side, it turns out to be extremely useful to carry information between the categories.

\subsection{}

We recall that $\p(n)_0 \cong \gl_n(\C)$ and we will use the set of simple roots $$\eps_2 - \eps_{1}, 
\ldots, \eps_n-\eps_{n-1},\VeraA{-\eps_{n-1}-\eps_n}$$ where the last root is odd and all others are even. Thus the dominant integral weights of $\p(n)$ are of the form $\lambda = \sum_i 
\lambda_i \eps_i$, where $\lambda_1 \leq \lambda_2 \leq \ldots\leq\lambda_n$ are integers. \InnaC{The set of dominant integral weights for $\p(n)$ will be denoted by $\Lambda_n$.}

 Let $L_n(\lam)$ be a simple module in $\F_n$ with highest weight $\lam$ whose highest weight space is purely even. \InnaC{All simple modules in $\F_n$ are of the form $L_n(\lam)$ or $\Pi L_n(\lam)$ for some $\lambda \in \Lambda_n$.}
 
 For each such weight $\lam$ we can construct a {\it cap diagram}: namely, we consider the integer line, and draw a \InnaC{black ball} in each position $\lam_i + (i-1)$, $i=1,2,\ldots, n$; the rest of the positions are empty. We then draw caps, where each cap has a bullet on the right end and an empty position on the left end. The cap diagram is drawn iteratively: at each step, we take a bullet which is not yet part of a cap, and draw a cap connecting this bullet with the closest empty position on its left, which is not yet part of any cap. 
 
 We will use the following terminology. If a cap $c'$ is sitting ``inside'' another cap $c$, we say that the $c'$ is internal to $c$; if there are no intermediate caps containing $c'$ and contained in $c$, we say that $c'$ is a successor of $c$.
 
 A cap $c$ will be called {\it maximal} if it is not internal to any cap.

 Let $x \in {\p(n)}_{\bar{1}}$ correspond to the root $2\eps_n$.
The first main result of this article, concerning the action of the $DS_x$ functor on simple modules, is as follows:
\begin{introtheorem}[See Theorem \ref{thrm:DS_clusters}, Corollary \ref{cor:DS_caps}]\label{introthrm:DS_clusters}

\mbox{}

The $\p(n-1)$-module $DS_{x}(L_n(\lam))$ is multiplicity free. Its composition factors can be explicitly described as simple modules $\Pi^{z(\lam, \mu)} L_{n-1}(\mu)$, where the cap diagram of $\mu$ is obtained by removing a single maximal cap from the cap diagram of $\lam$. \VeraA{Here we denote by $\Pi$ the functor $?\otimes \mathbb C^{0|1}$}. 

The parity $z(\lam, \mu) \in \Z/2\Z$ is given by $z(\lam, \mu) \equiv z \mod 2$, where $\lam_{n-z} + (n-z-1)$ is the rightmost end of the removed cap.
\end{introtheorem}

\begin{remark}
  A similar result for the general linear Lie superalgebra was proved in \cite{HeiWei14} using a similar technique. \VeraA{However, in contrast with the $\mathfrak{gl}(m|n)$-case, $DS_x(L_n(\lambda))$ may be not semisimple. For example, \InnaC{consider the case $n=2$ and the simple module $V_2$. Then $DS_x(V_2) \cong V_1$, which is indecomposable but not simple. Another example is} $n=3$ with $L_n(\lambda)$ being isomorphic to the simple ideal $\mathfrak{sp}(3)$ of matrices with zero supertrace. Then $DS_x(L_n(\lambda))$ is isomorphic to
 $\mathfrak{sp}(2)$ which is indecomposable but not simple $\p(2)$-module. } 
\end{remark}

In Section \ref{ssec:corollaries_DS}, we state some corollaries of this theorem, such as a criterion describing when the $\p(n-1)$-module $DS_x(L_n(\lam))$ is simple. 

\subsection{}
We next proceed to compute the superdimension of any simple finite-dimensional $\p(n)$-module. This is done by \InnaC{defining a subset of $\Lambda_n$ consisting of {\it worthy} weights. For any worthy weight $\lam$, we construct a rooted forest graph $F_{\lam}$.} If $\lam$ is not worthy, we show that $\sdim L_n(\lam)=0$. If $\lam$ is worthy, then $\sdim L_n(\lam)\neq 0$, and we give a simple combinatorial formula for $\sdim L_n(\lam)$ based on the rooted forest graph $F_{\lam}$. Below we elaborate on this result.

\mbox{}

To state the result on superdimensions, we will need additional terminology. 

A cap $c$ \InnaC{in a cap diagram} will be called {\it odd} if there is an odd number of caps internal to $c$, including $c$ itself. A weight $\lam \InnaC{\in \Lambda_n}$ will be called {\it worthy} if each cap $c$ \InnaC{in $d_{\lam}$} has at most one odd successor, and there is at most one maximal odd cap (such a cap will appear for worthy weights only when $n$ is odd).

If $\lam$ is worthy, we will construct a rooted forest $F_{\lam}$ corresponding to $\lam$ as follows. 
\begin{definition}\label{introdef:forest_lambda}
 Let $\lam$ be a worthy weight. We construct a rooted forest ${F}_{\lam}$ as follows.
  \begin{itemize}

  \item The nodes of ${F}_{\lam}$ are pairs of caps $(c_0, c_1)$ in the cap diagram of $\lam$, $c_0$ is even, and $c_1$ is the unique odd successor of $c_0$. If $n$ is odd, we also add a node $(c)$ corresponding to the unique maximal odd cap $c$ in the cap diagram of $\lam$.
 
  \item There is an edge from a node $v=(c_0, c_1)$ to a node $v'=(c'_0, c'_1)$ in ${F}_{\lam}$ if $c'_0$ is a successor of either $c_0$ or $c_1$. Similarly, if $n$ is odd and $v = (c)$ as above, then there is an edge from $v$ to $v' = (c'_0, c'_1)$ if $c'_0$ is a successor of $c$.
 \end{itemize}
\end{definition}
\begin{remark}
 \InnaC{This is a slightly different (but equivalent) version of Definition \ref{def:forest_lambda}.}
\end{remark}

We can now state our second main theorem.

\begin{theorem}[See Theorem \ref{thrm:sdim}]\label{introthrm:sdim}

\mbox{}

 Let $\lam \in \Lambda_n$ and let $L_n(\lam)$ be the corresponding simple module in $\F_n$ (with even highest weight vector, as before). 
 
 If the weight $\lam$ is not worthy (see Definition \ref{def:caps_prop}), then $$\sdim L_n(\lam)=0.$$
 
 If the weight $\lam$ is worthy, let $F_{\lam}$ be the corresponding rooted forest (as in Definition \ref{def:forest_lambda} above). Then $$\sdim L_n(\lam) = \frac{\abs{F_{\lam}}!}{F_{\lam}!}$$ where $\abs{F_{\lam}} = \lfloor\frac{n+1}{2}\rfloor$ is the number of nodes in the forest $F_{\lam}$, and $$F_{\lam}! = \prod_{v \text{ a node of } F_{\lam}} \sharp \text{ descendants of v in } F_{\lam}$$ is the forest factorial of $F_{\lam}$.
\end{theorem}

\begin{example}
For the weight $$\lam = \eps_3 + 3\eps_4 + 3\eps_5 + 3\eps_6 + 5\eps_7+ 7\eps_8 + 7\eps_9 + 7\eps_{10}$$ of $\p(10)$, the cap diagram will be 

$$ 
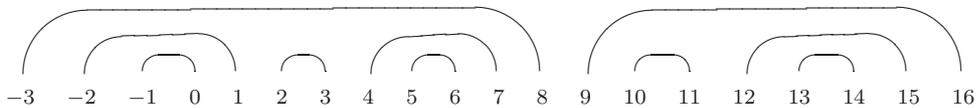
\xymatrix@=5pt{  &{} &{} &{} &{} &{} &{} &{} &{} &{} &{} &{} &{} &{} &{} &{}\\  
&\underset{-3}{} &\underset{-2}{} &\underset{-1}{} 
&\underset{0}{} \ar@{-} `u/6pt[l] `/6pt[l] &\underset{1}{} \ar@{-} `u/15pt[l] `/15pt[lll]  &\underset{2}{} &\underset{3}{}  \ar@{-} `u/6pt[l] `/6pt[l]&\underset{4}{} &\underset{5}{} &\underset{6}{} \ar@{-} `u/6pt[l] `/6pt[l]  &\underset{7}{ } \ar@{-} `u/15pt[l] `/15pt[lll] &\underset{8}{} \ar@{-} `u/25pt[l] `/25pt[lllllllllll]  &\underset{9}{} &\underset{10}{} &\underset{11}{} \ar@{-} `u/6pt[l] `/6pt[l]  &\underset{12}{} &\underset{13}{}  &\underset{14}{} \ar@{-} `u/6pt[l] `/6pt[l] &\underset{15}{} \ar@{-} `u/15pt[l] `/15pt[lll] &\underset{16}{} \ar@{-} `u/25pt[l] `/25pt[lllllll]  }$$ 

This is a worthy weight, with odd caps $(-1, 0), (2,3), (5,6), (10, 11), (13,14)$; the rest of the caps are even. The rooted forest will be $$ 
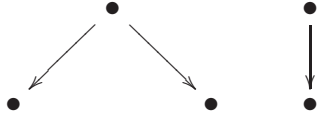
\xymatrix{  &{}  &{\bullet}    \ar[rd] \ar[ld] &{}    &{\bullet}\ar[d]  \\  &{\bullet} &{}  &{\bullet}      &{\bullet}  }$$ Hence $\sdim L_n(\lam) = \frac{5!}{3\cdot 1 \cdot 1 \cdot 2 \cdot 1} = 20$.
\end{example}

As a corollary, we recover the result of \cite{ES_KW} proving the Kac-Wakimoto conjecture for $\p(n)$: any module lying in a ``non-principal'' block of $\F_n$ (in the sense of \cite{ES_KW}) has superdimension zero.

\subsection{Acknowledgements}
I.E.-A. was supported by the ISF grant no. 711/18. V.S. was supported by NSF grant 1701532. Part of the work was carried out during the visit of V.S. to Ben Gurion University of the Negev, which was supported by the Faculty of Natural Sciences Distinguished Scientist Visitors Program and by the Center of Advanced Studies in Mathematics in Ben Gurion University.

The authors would like to thank Catharina Stroppel for her explanations on the cap diagrams.
\section{Preliminaries}\label{sec:prelim}
\subsection{General}

Throughout this paper, we will work over the base field $\mathbb C$, and all 
the categories considered will be $\C$-linear.

A {\it 
vector superspace} will be defined as a $\mathbb{Z}/2\mathbb{Z}$-graded vector 
space $V=V_{\bar 0}\oplus V_{\bar 1}$. The {\it parity} of a homogeneous vector 
$v \in V$ will be denoted by $p(v) \in \mathbb{Z}/2\mathbb{Z}=\{\bar 0, \bar 
1\}$ (whenever the notation $p(v)$ appears in formulas, we always assume that 
$v$ is homogeneous).

\subsection{The periplectic Lie superalgebra}\label{ssec:notn:periplectic}
\subsubsection{Definition of \VeraA{the} periplectic Lie superalgebra}\label{sssec:def_periplectic}
Let $n \in \Z_{>0}$, and let $V_n$ be an $(n|n)$-dimensional vector superspace 
equipped with a non-degenerate odd symmetric form
\begin{eqnarray}
\label{eq:form}
\omega:V_n\otimes V_n\to\mathbb C,\quad \omega(v,w)=\omega(w,v), 
\quad\text{and}\quad 
\omega(v,w)=0\,\,\text{if}\,\,p(v)=p(w).
\end{eqnarray}

Then $\operatorname{End}_{\mathbb C}(V_n)$ inherits the structure of a vector 
superspace from $V_n$. We denote by $\mathfrak{p}(n)$ the Lie superalgebra of 
all 
$X\in\operatorname{End}_{\mathbb C}(V_n)$ preserving $\omega$, i.e. satisfying  
$$\omega(Xv,w)+(-1)^{p({X})p(v)}\omega(v,Xw)=0.$$

\begin{remark}\label{rmk:basis}
Choosing dual bases $v_1, v_2, \ldots, v_n$ in $V_{\bar{0}, n}$ and $v_{1'}, 
v_{2'},\ldots v_{n'}$ in $V_{\bar{1}, n}$, we can write the matrix of $X\in 
\mathfrak{p}(n)$ as
$\left(\begin{smallmatrix}A&B\\C&-A^t\end{smallmatrix}\right)$
where $A,B,C$ are $n\times n$ matrices such that $B^t=B,\, C^t=-C$.
\end{remark}

We will also use the triangular decomposition 
${\p(n)} \cong 
{\p(n)}_{-1} \oplus {\p(n)}_0 \oplus {\p(n)}_1$ where $${\p(n)}_0 \cong \gl(n), \;\; 
{\p(n)}_{-1} 
\cong \Pi\wedge^2(\C^n)^*, \;\; {\p(n)}_1 \cong \Pi S^2\C^{n}.$$
Then the action of $\p(n)_{\pm 1}$ on any $\p(n)$-module is ${\p(n)}_0$-equivariant.

\subsubsection{Weights for the periplectic 
superalgebra}\label{sssec:notn:weight_p_n}
The integral weight lattice for $\mathfrak{p}(n)$ will be $span_{\mathbb 
Z}\{\eps_i\}_{i=1}^n$.

\begin{itemize}[label={$\star$}]
 \item We fix a set of simple roots $\eps_2 - \eps_{1}, 
\ldots, \eps_n-\eps_{n-1},\VeraA{-\eps_{n-1}-\eps_n}$, the \VeraA{last} root is odd and all others are even.

Hence the dominant integral weights will be given by $\lambda = \sum_i 
\lambda_i \eps_i$, where $\lambda_1 \leq \lambda_2 \leq \ldots\leq\lambda_n$.
\item We fix an order on the weights of $\p(n)$: for weights $\mu, \lambda$, we 
say that $\mu \geq \lambda$ if $\mu_i \leq \lambda_i$ for each $i$. The set of all dominant integral weights for $\p(n)$ will be denoted by $\Lambda_n$.

\begin{remark}
 It was shown in \cite[Section 3.3]{BDE:periplectic} that \InnaC{the order} $\leq$ corresponds to a 
highest-weight structure on the category of finite-dimensional representations 
of $\p(n)$.\VeraA{ Note that in the cited paper we use slightly different set of simple roots
 $-\eps_1-\eps_2,\eps_1 - \eps_{2}, \ldots, \eps_{n-1}-\eps_{n}$.} 
\end{remark}

\item
The simple finite-dimensional representation of $\p(n)$ corresponding to the 
weight $\lambda$ whose highest weight vector is {\it even} will be denoted by 
$L_n(\lambda)$.
\begin{example}
 Let $n\geq 2$. The natural representation $V_n$ of $\p(n)$ has highest weight $-\eps_1$, 
with odd highest-weight vector; hence $V_n \cong \Pi L_n(-\eps_1)$. The 
representation $\bigwedge^2 V_n$ has highest weight $-2\eps_1$, and the 
representation $S^2 V_n$ has highest weight $-\eps_1 - \eps_2$; both have even 
highest weight vectors, so $$\wedge^2 V_n \twoheadrightarrow L_n(-2\eps_1), \;\; L_n(-\eps_1 - \eps_2) \hookrightarrow S^2 V_n .$$
\end{example}
\item Set $\rho^{(n)} = 
\sum_{i=1}^n (i-1)\eps_i$, and for any weight 
$\lambda$, denote $$\bar{\lambda}  = \lambda+\rho^{(n)} .$$

\end{itemize}

\subsubsection{Representations of 
\texorpdfstring{$\p(n)$}{p(n)}}\label{sssec:notn:repr_periplectic}

We denote by $\F_n$ the category of 
finite-dimensional representations of
$\p(n)$ whose restriction to $\p(n)_{\bar 0} \cong 
\mathfrak{gl}(n)$ integrates to an action of $GL(n)$.

By definition, the morphisms in $\F_n$ will be {\it 
grading-preserving} 
$\p(n)$-morphisms,
i.e., $\operatorname{Hom}_{\F_n}(X,Y)$ is a vector space and not a 
vector superspace. This is important in order to ensure that the category
$\F_n$ be abelian.

The category $\F_n$ is not semisimple. In fact, this 
category is a highest-weight category; \InnaC{more about the highest-weight structure can be found in \cite{BDE:periplectic}.}

\subsubsection{Weight diagrams and arrows}
The following notation has been introduced in \cite{BDE:periplectic}. 

For $\lambda$ a dominant weight we define the map 
\begin{eqnarray*}
f_\lambda:\mathbb{Z} \to \{0,1 \}\quad\text{as}\quad
f_\lambda(i)&=&
\begin{cases}1& \text{if $i\in \{\bar{\lambda}_j,\, j=1, \ldots, n\}$},\\
0& \text{else}.
\end{cases}
\end{eqnarray*}
 The corresponding  {\it weight diagram} $d_\lambda$ is the  labeling of the integer line by symbols $\bullet$ (``black ball'') and $\circ$ (``empty'') such that $i$ has label $\bullet$ if $f(i) =1$, and label $\circ$ otherwise.

\begin{notation}\label{notn:g_lam}
 For $\lambda \in \Lambda_n$ we define the function $g_\lambda:\mathbb Z\to\{-1,1\}$
by setting $$g_\lambda(i)=(-1)^{f_\lambda(i) +1}.$$

\end{notation}
So $g_\lambda(i)=1$ if $d_\lambda$ has a black ball at the $i$-th position and $g_\lambda(i)=-1$ otherwise.

\begin{notation}
 For any $i<j$ set 
$r_{\lam}(j,i)\;=\;\sum_{s=i}^{j-1}g_{\lambda}(s)$.
\end{notation}

As in \cite[Section 6.2]{BDE:periplectic}, in the diagram $d_{\lam}$ we will draw a solid\footnote{In this paper we do not use any other types of arrows, but in \cite{BDE:periplectic} ``dual'' dashed arrows were introduced as well.} arrow from position $j$ to position $i<j$ if $f_\lam(j)=1=g_\lambda(j)$, and if $$r_{\lam}(j,i)=0,\,\, \text{  and  } \,\forall \, i<s<j, \; r_{\lam}(j,s)\geq 0.$$

\begin{example}\label{ex:wt1}
 Let $n = 6$, $\lambda = (0,0, 1, 1,1,3)$. The diagram $d_{\lam}$ is given by

  $$ 
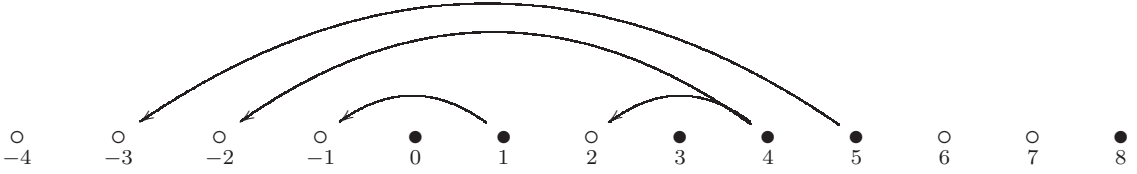
\xymatrix@=20pt{&{} &{} &{} &{}&{}  &{}&{}&{}&{}&{}  &{} \\  &\underset{-4}{\circ}  &\underset{-3}{\circ} &\underset{-2}{\circ}  &\underset{-1}{\circ}  &\underset{0}{\bullet}  &\underset{1}{\bullet} \ar@/_1.3pc/[luld]  &\underset{2}{\circ}  &\underset{3}{\bullet} &\underset{4}{\bullet} \ar@/_1.3pc/[luld] \ar@/_3.3pc/[lllullld] &\underset{5}{\bullet} \ar@/_4.2pc/[llllulllld] &\underset{6}{\circ} &\underset{7}{\circ}  &\underset{8}{\bullet} } $$
 and all other positions on the integer line are empty.
\end{example}

 \begin{definition}\label{def:cluster}
  A (black) cluster in a weight diagram $d_{\alpha}$ is a sequence of consecutive black balls: 
  
\resizebox{\textwidth}{!}
{
$d_\alpha = \xymatrix{&\underset{i-1}{\circ}  &\underset{i}{\bullet}  &\underset{i+1}{\bullet}  &{\ldots} & &\underset{j-1}{\bullet} &\underset{j}{\bullet} &\underset{j+1}{\circ}}$
}

In other words, it is a segment in of the form $[i, j]$, $i<j$ such that $$f_{\alpha}(i-1)=0, \; f_{\alpha}(i)=f_{\alpha}(i+1)=\ldots =f_{\alpha}(j-1)=f_{\alpha}(j)=1, \; f_{\alpha}(j+1)=0.$$

Position $i$ is called the beginning of the cluster, and position $j$ is called the end of the cluster.
 \end{definition}
 \begin{notation}
 Let $\lam \in \Lambda_n$. Consider the solid arrows in the diagram $d_{\lam}$. We will call a solid arrow {\it maximal} if there is no solid arrow above it; in other words, a solid arrow from $j$ to $i$ is called maximal if there is no solid arrow from $k$ to $l$ where $l\leq i$, $k\geq j$ and $(k, l)\neq (j, i)$.
\end{notation}
 \subsubsection{Cap diagrams}\label{sssec:cap_diag}
 Consider the weight diagram $d_{\lambda}$ of $\lambda$. Instead of drawing arrows, we can draw a cap diagram on the integer line $\Z$ iteratively as follows: for any pair of positions $(i, j)$, $i<j$ 
such that $f_{\lam} (i) =0$, $f_\lam(j) =1$, we draw a cap connecting these 
(oriented from $j$ to $i$) if all the positions $i <l <j$ are already part of 
some cap. 

Clearly, every black ball in $d_{\lam}$ will be the right end of exactly one cap. The weight diagram $d_{\lam}$ can be uniquely determined from the cap diagram (by abuse of notation, the cap diagram is also denoted $d_{\lam}$).

\begin{definition}\label{def:caps_poset}
 
 \mbox{}
 \begin{itemize}
  \item A cap $(i, j)$ is called {\it internal} to a cap $(i', j')$ if $i' \leq 
i < j \leq j'$. We denote: $(i, j) \lhd (i', j')$.
\item A cap $(i, j)$ is called {\it maximal} if it is not internal to any other cap. 
  \item A cap $(i, j)$ is called {\it a successor} of a cap $(i', j')$ if $(i, j) \lhd (i', j')$ and there is no cap $(i'', j'')$ such that $(i, j) \lhd (i'', j'') \lhd (i', j')$. 
 \end{itemize}

\end{definition}
\begin{example}\label{ex:wt1_caps}

Consider the weight $\lambda = \eps_1 + \eps_2+ 3\eps_3+5\eps_4+5\eps_5+5\eps_6$ for $\p(6)$, \InnaC{as in Example \ref{ex:wt1}}. Here we draw the cap diagram for $\lambda$ on top of the weight diagram $d_{\lam}$:  

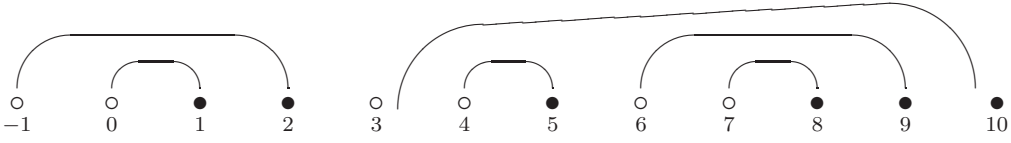
$$ \xymatrix@=20pt{  &{} &{} &{} &{} &{} &{} &{} &{} &{} &{} &{} &{} &{} &{} &{}\\  &\underset{-1}{\circ} &\underset{0}{\circ} 
&\underset{1}{\bullet}  \ar@{-} `u/10pt[l] `/10pt[l] &\underset{2}{\bullet}  \ar@{-} `u/20pt[lll] `/20pt[lll] &\underset{3}{\circ} &\underset{4}{\circ} &\underset{5}{\bullet}  \ar@{-} `u/10pt[l] `/10pt[l] &\underset{6}{\circ} &\underset{7}{\circ } &\underset{8}{\bullet} \ar@{-} `u/10pt[l] `/10pt[l]  &\underset{9}{\bullet}  \ar@{-} `u/20pt[lll] `/20pt[lll] &\underset{10}{\bullet} \ar@{-} `u/40pt[l] `/40pt[lllllll] }$$

The partial order on the caps in this diagram is: $$(0,1) \lhd (-1,2), \;\; (4,5) \lhd (3,10),\; (7,8) \lhd (6,9) \lhd (3,10).$$ The maximal caps here are $(-1,2)$ and $(3,10)$. The successors of the cap $(3,10)$ are $(4,5)$, $(6,9)$.
\end{example}
\VeraA{
\begin{remark} Every solid arrow goes from the left end of a cap to
  the right end of its successor cap. In particular, the total number
  of solid arrows equals $n$ minus the number of maximal caps.
  \end{remark}}
\begin{lemma}
 Let $(i, j)$ be a cap in the cap diagram $d_{\lam}$. Then exactly one of the following is true:
 \begin{itemize}
  \item We have $i+1=j$.
  \item There is a solid arrow from $j$ to $i+1$, and this is the longest solid arrow originating in $j$.
 \end{itemize}

\end{lemma}

\begin{proof}
\InnaC{
First of all, if $i+1=j$ then clearly there is no solid arrow from $j$ to $i+1$.
 Assume $i+1 \neq j$.} By the construction of the cap diagram, we have: $$\forall i+1\leq l\leq j-1, \;\;\InnaC{r_{\lam}(j,l)=}\sum_{s=l}^{j-1} g_\lam(s)\,\InnaC{\geq } \, 0, \;\InnaC{r_{\lam}(j,i+1)=}\sum_{s=i+1}^{j-1} g_\lam(s)=0, \; \InnaC{r_{\lam}(j,i)<0}$$ Hence the statement follows. 
\end{proof}

\begin{corollary}\label{cor:max_cap_arrow}
 Let $(i, j)$ be a maximal cap in the cap diagram of $d_{\lam}$. Then there is a solid arrow from $j$ to $i+1$, and this solid arrow is maximal.
\end{corollary}

\subsubsection{Tensor Casimir and translation functors}

Consider the following natural endomorphism ${\Omega}^{(n)}$ of the endofunctor $(-) 
\otimes V_n$ on $\mathcal{F}_n$.

\VeraA{ Recall that $\p_n$ is the set of fixed points of the involutive automorphism $\sigma$ of $\gl(V_n)$.
We consider the $\p_n$-equivariant decomposition:
$$\gl(V_n) \cong\p_n \oplus \p_n^*$$
where $\p_n^*$ is the eigenspace of $\sigma$ with eigenvalue $-1$.
Both $\p_n$ and $\p_n^*$ are maximal isotropic subspaces with respect to the invariant symmetric form on $\gl(V_n)$ and hence this form defines
a non-degenerated pairing $\p_n^*\otimes\p_n\to\C$.}

We begin by taking the orthogonal $\p_n$-equivariant decomposition
$$\gl(V_n) \cong\p_n \oplus \p_n^*$$ with respect to 
the form $$tr:=ev\circ \sigma_{V_n, V_n^*}: \gl(V_n) \cong V_n \otimes V_n^* 
\to 
\C.$$

\begin{definition}[Tensor Casimir]\label{def:Casimir}
For any $M \in \mathcal{F}_n$, let ${\Omega}_M$ be the composition
$$ V_n \otimes M \xrightarrow{\id \otimes coev \otimes \id} V_n \otimes \p(n)^* 
\otimes \p(n) \otimes M \xrightarrow{i_* \otimes \id}  V_n \otimes \gl(V_n) 
\otimes \p(n) \otimes M   \xrightarrow{act \otimes act} 
V_n \otimes M $$ where $i_*: \p(n)^* \to \gl(V_n)$ is the $\p(n)$-equivariant embedding defined above.
\end{definition}
\begin{definition}[Translation functors]\label{def:transl_functors}
 For $k \in \C$, we define a functor 
${\Theta'}^{(n)}_k:\mathcal{F}_n\to\mathcal{F}_n$ as the functor $\Theta^{(n)}= ( - )\otimes V_n$ 
followed by the projection onto the generalized $k$-eigenspace for 
${\Omega}^{(n)}$, i.e. 
\begin{eqnarray}
\label{thetak}
{\Theta'}^{(n)}_k(M):=\bigcup_{m>0}\Ker ({\Omega}^{(n)} 
-k\operatorname{Id})^m_{|_{M\otimes V_n}}
\end{eqnarray}
and set ${\Theta}^{(n)}_k:=\Pi^k{\Theta'}^{(n)}_k$ in case $k\in\mathbb{Z}$ (it was proved in \cite{BDE:periplectic} that $ \forall k \notin \Z, \; {\Theta}^{(n)}_k \cong 0$). 
\end{definition}

The functors $\Theta_k$ are exact, since $ - \otimes V_n$ is an exact functor, and we have the following result, proved in \cite{BDE:periplectic}.
\begin{theorem}[See \cite{BDE:periplectic}.]\label{old_thrm:transl_proj_simples}

Let $L, L'$ be non-isomorphic simple modules in $\F_n$. Let $i\in \Z$.
\begin{enumerate} 
\item The module $\Theta_i L$ is multiplicity free.
\item The modules $\Theta_i(L)$ and $\Theta_i(L')$ have no common simple subquotients (that is, their sets of composition factors are disjoint).
\end{enumerate}

\end{theorem}
For more details on the structure of $\F_n$, see  
\cite{BDE:periplectic}. 
\begin{lemma}\label{lem:transimple} Assume that
  $\Theta_i(L_n(\lambda))\neq 0$. Then
  \begin{enumerate}
  \item\label{itm:transimple_1} $f_\lambda(i)=1$, $f_\lambda(i-1)=0$.
    \item\label{itm:transimple_2} Let $d_{\lambda'}$ be obtained from $d_\lambda$ by moving
      $\bullet$ from position $i$ to position $i-1$. Then
      $[\Theta_i(L_n(\lambda)):\Pi^{i+1}L_n(\lambda')]=1$.
      \item\label{itm:transimple_3} If $[\Theta_i(L_n(\lambda)):\Pi^z L_n(\mu)]\neq 0$ for \InnaD{some $z\in \{0,1\}$} and
        $\mu\neq\InnaD{\lambda'}$, then $f_\mu(i)\neq 0$ or $f_\mu(i-1)\neq 1$.
        \item\label{itm:transimple_4} \InnaD{If $[\Theta_i(L_n(\lambda)):\Pi^z L_n(\mu)]\neq 0$ for \InnaD{some $z\in \{0,1\}$} and
         $f_\mu(i)= f_\mu(i-1)=0$, then $f_{\mu}(s) = f_{\lam}(s)$ for any $s\leq i-1$.} 
        \end{enumerate}
  \end{lemma}
\begin{proof} 
\InnaD{The statement in \eqref{itm:transimple_1}} is already proven in \cite[\InnaD{Corollary 8.2.2}]{BDE:periplectic}.

To prove \InnaD{\eqref{itm:transimple_2},} recall that we have an exact sequence
$$0\to \Pi^{i+1}\nabla(\lambda')\to \Theta_i(\nabla(\lambda))\to M\to 0$$
where either $M=0$ or $M=\nabla(\lambda'')$ where $d_{\lambda''}$ is
obtained from $d_\lambda$ by moving $\bullet$ from $i$ to $i+1$
if it is possible. Therefore we have an embedding $\InnaD{\Pi^{i+1}} L_n(\lambda')\to\Theta_i(\nabla(\lambda))$.
On the other hand,  all composition factors \InnaD{(up to change of parity)} $L_n(\nu)$ of $\nabla(\lambda)$ satisfy
the condition $\nu=\lambda+\sum a_{ij}(\varepsilon_i+\varepsilon_j)$
for some $a_{ij}\in\mathbb N$. That ensures that
$[L_n(\nu)\otimes  V:L_{n}(\lambda')]=0$ unless $\nu=\lambda$.
Hence $[\Theta_i(L_n(\lambda)):\InnaD{\Pi^{i+1}} L_n(\lambda')]=1$.

To show \InnaD{\eqref{itm:transimple_3},} assume the opposite, i.e., $f_\mu(i)= 0$ \InnaD{and}
$f_\mu(i-1)=1$. Let $d_\nu$ be obtained from $d_\mu$ by moving black
ball from $i-1$ to $i$. Then by  \InnaD{\eqref{itm:transimple_2},} we have
$[\Theta_i(L_n(\nu)):\InnaD{\Pi^{i+1}} L_n(\mu)]=1$. Therefore $L_n(\mu)$ \InnaD{(up to change of parity)} appears as a
composition \InnaD{factor} in both $\Theta_i(L_n(\lambda))$ and
$\Theta_i(L_n(\nu))$. This contradicts Theorem
\ref{old_thrm:transl_proj_simples} (2). 

\InnaD{The statement in \eqref{itm:transimple_4} is proved in the same methods as in the proof of \cite[\InnaD{Corollary 8.2.2}]{BDE:periplectic}. Assume that $[\Theta_i(L_n(\lambda)):\Pi^z L_n(\mu)]\neq 0$ for \InnaD{some $z\in \{0,1\}$} and that $f_\mu(i)= f_\mu(i-1)=0$. Denote by $P_n(\lam), P_n(\mu)$ the projective covers of $L_n(\lambda), L_n(\mu)$ respectively. Then 

\begin{align*}
&\dim \Hom_{\p(n)}(\Theta_{i+1} P_n(\mu),L_n(\lam)) = \dim \Hom_{\p(n)}(P_n(\mu),\Theta_i(L_n(\lam))) =\\
&= [\Theta_i(L_n(\lam)):\Pi^z L_n(\mu)]\neq 0 
\end{align*} 

Now, by \cite[Lemma 7.2.3]{BDE:periplectic}, the statement of \eqref{itm:transimple_4} follows.} 
\end{proof}

\subsubsection{Blocks}\label{sssec:blocks}
 There are $2(n+1)$ blocks in the category $\F_n$ has blocks. These blocks are in bijection with the set
$\{ -n, -n+2 ,\ldots, n-2, n\}\times\{+,-\}$.

We have a decomposition 
$$\F_n=\bigoplus_{k\in  \{ -n, -n+2 ,\ldots, n-2, n\}}\left( \F_n \right)^+_k\oplus \bigoplus_{k\in  \{ -n, -n+2 ,\ldots, n-2, n\}}\left( \F_n \right)^-_k,$$
where the functor $\Pi$ (parity change) induces an equivalence $\left( \F_n \right)_k^+ \cong \left( \F_n \right)_k^-$. 
Hence we may define {\it up-to-parity blocks} $$\F^k_n := \left( \F_n \right)_k^+ \oplus \left( \F_n \right)_k^-.$$ The block $\InnaC{\F^k_n}$ contains all simple modules $L(\lambda)$ with $$\sum_i (-1)^{\bar{\lam}_i}=k$$
By abuse of terminology, we will just call these ``blocks'' throughout the paper.

\begin{theorem}[See \cite{BDE:periplectic}.]\label{old_thrm:blocks_action}
Let $i \in \mathbb Z$, $k \in \{ -n, -n +2 , \ldots, n-2, n\}$. Then we have 
$$\Theta_i \VeraA{\F_n^k}\subset \begin{cases} \F_n^{k+2} \,\,\text{if}\,\,i\,\, \text{is odd}\\
\mathcal{F}_n^{k-2} \,\,\text{if}\,\,i\,\, \text{is even}\\
\end{cases}$$
\end{theorem}

\subsection{The Duflo-Serganova functor}\label{sec:DS_functor}

Let $n \geq 2$, and let $x\in \p(n)$ be an odd element such that $[x,x]=0$. Let $s := \rk(x)$. We define the following correspondence of vector superspaces:
\begin{definition}[See \cite{DufloSer:functor}]\label{def:DS}
 Let $M \in \F_n$, and consider the complex $$\Pi M \xrightarrow{x} M \xrightarrow{x} \Pi M$$ We define 
$ {DS}_x(M)$ to be the homology of this complex. 
\end{definition}
 
The vector superspace $\p_x:=DS_x\p(n)$ is naturally equipped with a Lie superalgebra structure.\VeraA{ One can check by direct computations that
  $\p_x$ is isomorphic to $\p(n-s)$ where $s$ is the rank of $x$. The above correspondence defines an SM-functor $DS_x:\F_n \to \F_{n-s}$, called the {\it Duflo-Serganova functor}}. Such functors were introduced in \cite{DufloSer:functor}.

  The following lemmata are used extensively throughout this paper.
\InnaD{
\begin{lemma}[See \cite{DufloSer:functor}]
 Given a short exact sequence $$0 \to M_1 \xrightarrow{f} M_2 \xrightarrow{g} M_3 \to 0$$ in $\F_n$, we have an exact sequence
 $$ 0 \to E \to DS_x(M_1) \xrightarrow{DS_x(f)} DS_x(M_2) \xrightarrow{DS_x(g)}  DS_x(M_3) \to \Pi E \to 0$$ for some $E \subset DS_x(M_1)$ in $\F_{n-s}$. 
 
\end{lemma}

In particular, if $L$ is a simple composition factor of $DS_x(M_2)$, then it is a simple composition factor of $DS_x(M_1)$ or of $DS_x(M_3)$.}
  
\begin{lemma}[See \cite{ES_KW}]\label{cor:DS_translation_commute}
 The functor $DS_x$ commutes with translation functors, that is we have a natural isomorphism of functors $$DS_x {\Theta}^{(n)}_k \stackrel{\sim}{\longrightarrow}
\Theta^{(n-s)}_k DS_x$$ for any $k \in \Z$.

\end{lemma}

\section{The Duflo-Serganova functor: main theorem}\label{sec:DS_main}

Let $x_n\in \p(n)_1, x_n\neq 0$ be an odd element corresponding to the root $2\eps_n$. Then $[x_n,x_n]=0$ and we may define a Duflo-Serganova functor $$DS_{x_n}: \F_n \to  \F_{n-1} $$ as in Section \ref{sec:DS_functor}.

Throughout this section, we will write $DS=DS_{x_n}$ for short.
\subsection{Statement of the theorem}
  Let $ \lam \in \Lambda_n$.
  
  As before, we denote by $L_n(\lam)$ the simple finite-dimensional integrable $\p(n)$-module with an even highest weight vector of weight $\lam$. We consider the simple subquotients of $DS(L_n(\lam))$ in $\F_{n-1}$.

 \begin{theorem}\label{thrm:DS_clusters}
 Let $ \lam \in \Lambda_n$ and $\mu \in\Lambda_{n-1}$.
 
 The following are equivalent:
 \begin{enumerate}
  \item $[DS(L_n(\lam)):\Pi^z L_{n-1}(\mu)] \neq 0$ for some $z \in \Z$.
  \item\label{itm:DS_clusters2} The diagram $d_{\mu}$ is obtained by removing one black ball from position $i$ in $d_{\lam}$, where $i$ satisfies the Initial Segment Condition: $$\forall j> i+1, \; r_{\lam}(j, i+1)\geq 0.$$
  In other words, $\InnaC{f_{\lam}(i+1)}=0$ and there is no solid arrow in $d_{\lam}$ ending in $i+1$. 
  
 \end{enumerate}
 Furthermore, if these conditions hold, then $$[DS(L_n(\lam)):\Pi^z L_{n-1}(\mu)]=1$$ \InnaA{where $i=\bar{\lam}_{n-z}$ (that is, $0\leq z\leq n-1$ and $n-z$ is the sequential number of the removed black ball (counting from the left)).} 
\end{theorem}

\begin{remark}
 For any position $i$ satisfying the Initial Segment Condition \InnaC{\eqref{itm:DS_clusters2}}, and any $j\geq i+1$, we have: in the segment $[i+1, j]$ in $d_{\lam}$ the number of empty positions is greater or equal to the number of black balls in that segment. 
\end{remark}

\begin{proof}[Proof of Theorem \ref{thrm:DS_clusters}]
 The proof goes as follows:
 \begin{enumerate}
  \item Assume $[DS(L_n(\lam)):DS(L_{n-1}(\mu))] \neq 0$. 
  \begin{itemize}
   \item First, we prove: $$f_{\mu}(i-1)=0, \, f_{\mu}(i)=1 \; \Longrightarrow \; f_{\lam}(i-1)=0, \, f_{\lam}(i)=1.$$ In other words, the clusters in $d_{\mu}$ begin in the same positions as in $d_{\lam}$. This is proved in Lemma \ref{lem:DS_cluster_start}.
   \item Secondly, we prove: $$\forall i,\;  f_{\lam}(i)\geq f_{\mu}(i).$$ In other words, if a position in $d_{\lam}$ was empty, so is the corresponding position in $d_{\mu}$. This is proved in Proposition \ref{prop:DS_cluster_growth}.
  \end{itemize}
Hence we conclude: if $[DS(L_n(\lam)):L_{n-1}(\mu)]\neq 0$ then $d_{\mu}$ is obtained from $d_{\lam}$ by removing one black ball from the end of some cluster. 
  \item Next, we prove Proposition \ref{prop:DS_nonremovable}, stating that black balls which do not satisfy the Initial Segment Condition \eqref{itm:DS_clusters2} cannot be removed.
  
  \item We prove Proposition \ref{prop:existence}, which completes the proof of the Theorem.
 \end{enumerate}

\end{proof}
 
\begin{example}\label{ex:DS_computation}
For the weight $$\lam = \eps_3 + 3\eps_4 + 3\eps_5  + 6\eps_6+ 8\eps_7 + 8\eps_8 + 8\eps_9$$ of $\p(9)$, the arrow diagram will be 

$$ 
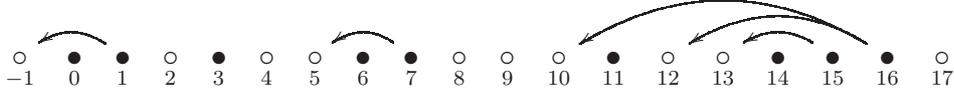
\xymatrix@=5pt{  &{} &{} &{} &{} &{} &{} &{} &{} &{} &{} &{} &{} &{} &{} &{}\\  
 &\underset{-1}{\circ} 
&\underset{0}{\bullet} &\underset{1}{\bullet} \ar@/_0.8pc/[ulld] &\underset{2}{\circ} &\underset{3}{\bullet}  &\underset{4}{\circ} &\underset{5}{\circ} &\underset{6}{\bullet}  &\underset{7}{\bullet} \ar@/_0.8pc/[ulld]  &\underset{8}{\circ}  &\underset{9}{\circ} &\underset{10}{\circ} &\underset{11}{\bullet}  &\underset{12}{\circ} &\underset{13}{\circ}  &\underset{14}{\bullet}  &\underset{15}{\bullet} \ar@/_0.8pc/[ulld] &\underset{16}{\bullet} \ar@/_1.3pc/[ulllld] \ar@/_1.8pc/[ulllllld] &\underset{17}{\circ}  }$$

Then the simple factors of $DS_{x_9}(L_9(\lam))$ are $\Pi L_8(\mu_1), L_8(\mu_2), L_8(\mu_3), L_8(\mu_4)$ where
\begin{align*}
 &\mu_1 =  2\eps_2 + 4\eps_3+ 4\eps_4 + 7\eps_5  + 9\eps_6+ 9\eps_7 + 9\eps_8, \\
 &\mu_2 =   4\eps_3+ 4\eps_4 + 7\eps_5  + 9\eps_6+ 9\eps_7 + 9\eps_8, \\ 
 &\mu_3=   \eps_3+ 3\eps_4 + 7\eps_5  + 9\eps_6+ 9\eps_7 + 9\eps_8, \\ 
 &\mu_4 =   \eps_3+ 3\eps_4 + 3\eps_5  + 6\eps_6+ 8\eps_7 + 8\eps_8. \\
\end{align*}
are weights in $\Lambda_8$ with arrow diagrams 
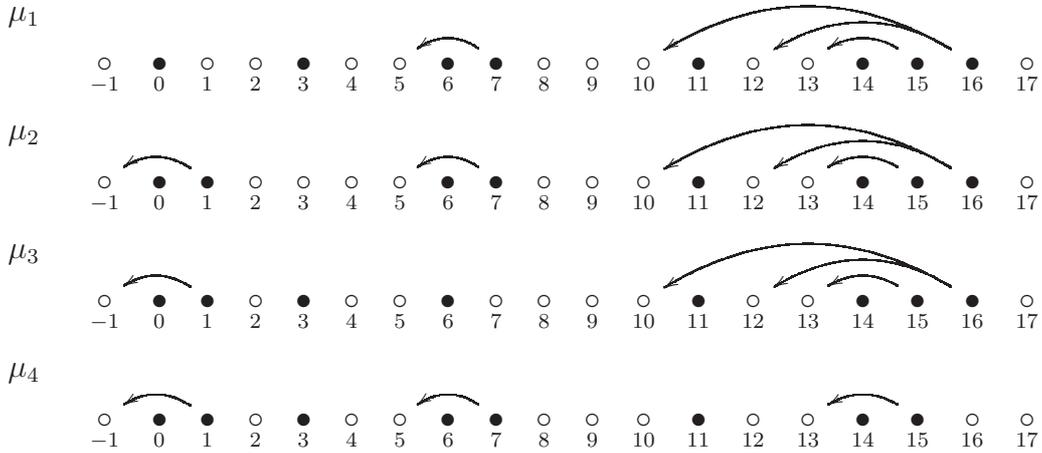
\begin{align*}
 \mu_1  \; &\xymatrix@=5pt{  &{} &{} &{} &{} &{} &{} &{} &{} &{} &{} &{} &{} &{} &{} &{}\\  
 &\underset{-1}{\circ} 
&\underset{0}{\bullet} &\underset{1}{\circ}  &\underset{2}{\circ} &\underset{3}{\bullet}  &\underset{4}{\circ} &\underset{5}{\circ} &\underset{6}{\bullet}  &\underset{7}{\bullet} \ar@/_0.8pc/[ulld]  &\underset{8}{\circ}  &\underset{9}{\circ} &\underset{10}{\circ} &\underset{11}{\bullet}  &\underset{12}{\circ} &\underset{13}{\circ}  &\underset{14}{\bullet}  &\underset{15}{\bullet} \ar@/_0.8pc/[ulld] &\underset{16}{\bullet} \ar@/_1.3pc/[ulllld] \ar@/_1.8pc/[ulllllld] &\underset{17}{\circ}  } \\ 
\mu_2 \; &\xymatrix@=5pt{  &{} &{} &{} &{} &{} &{} &{} &{} &{} &{} &{} &{} &{} &{} &{}\\  
 &\underset{-1}{\circ} 
&\underset{0}{\bullet} &\underset{1}{\bullet} \ar@/_0.8pc/[ulld] &\underset{2}{\circ} &\underset{3}{\circ}  &\underset{4}{\circ} &\underset{5}{\circ} &\underset{6}{\bullet}  &\underset{7}{\bullet} \ar@/_0.8pc/[ulld]  &\underset{8}{\circ}  &\underset{9}{\circ} &\underset{10}{\circ} &\underset{11}{\bullet}  &\underset{12}{\circ} &\underset{13}{\circ}  &\underset{14}{\bullet}  &\underset{15}{\bullet} \ar@/_0.8pc/[ulld] &\underset{16}{\bullet} \ar@/_1.3pc/[ulllld] \ar@/_1.8pc/[ulllllld] &\underset{17}{\circ}  } \\ 
\mu_3  \;  &\xymatrix@=5pt{  &{} &{} &{} &{} &{} &{} &{} &{} &{} &{} &{} &{} &{} &{} &{}\\  
 &\underset{-1}{\circ} 
&\underset{0}{\bullet} &\underset{1}{\bullet} \ar@/_0.8pc/[ulld] &\underset{2}{\circ} &\underset{3}{\bullet}  &\underset{4}{\circ} &\underset{5}{\circ} &\underset{6}{\bullet}  &\underset{7}{\circ}   &\underset{8}{\circ}  &\underset{9}{\circ} &\underset{10}{\circ} &\underset{11}{\bullet}  &\underset{12}{\circ} &\underset{13}{\circ}  &\underset{14}{\bullet}  &\underset{15}{\bullet} \ar@/_0.8pc/[ulld] &\underset{16}{\bullet} \ar@/_1.3pc/[ulllld] \ar@/_1.8pc/[ulllllld] &\underset{17}{\circ}  } \\ 
\mu_4  \;  &\xymatrix@=5pt{  &{} &{} &{} &{} &{} &{} &{} &{} &{} &{} &{} &{} &{} &{} &{}\\  
 &\underset{-1}{\circ} 
&\underset{0}{\bullet} &\underset{1}{\bullet} \ar@/_0.8pc/[ulld] &\underset{2}{\circ} &\underset{3}{\bullet}  &\underset{4}{\circ} &\underset{5}{\circ} &\underset{6}{\bullet}  &\underset{7}{\bullet} \ar@/_0.8pc/[ulld]  &\underset{8}{\circ}  &\underset{9}{\circ} &\underset{10}{\circ} &\underset{11}{\bullet}  &\underset{12}{\circ} &\underset{13}{\circ}  &\underset{14}{\bullet}  &\underset{15}{\bullet} \ar@/_0.8pc/[ulld] &\underset{16}{\circ}  &\underset{17}{\circ}  }.
\end{align*}
\end{example}

We also give a formulation of the theorem using cap diagrams, which will suit our needs better when computing superdimensions.

\InnaC{The following is a rephrasing of} the statement of Theorem \ref{thrm:DS_clusters}, using \InnaC{Corollary} \ref{cor:max_cap_arrow}:
\begin{corollary}\label{cor:DS_caps}
 Let $\lam \in \Lambda_n$, $x\in \p(n)_1$, $\mu\in \Lambda_{n-1}$. The following are equivalent:
 \begin{enumerate}
  \item $[DS(L_n(\lam)):\InnaC{\Pi^z} L_{n-1}(\mu)] \neq 0$ \InnaC{for some $z\in \Z$}.
  \item The diagram $d_{\mu}$ is obtained \InnaC{from} $d_{\lam}$ by removing one maximal cap.
  
 \end{enumerate}
 Furthermore, if these conditions hold, then $[DS(L_n(\lam)):\Pi^z L_{n-1}(\mu)]=1$, where \InnaC{position} $\bar{\lambda}_{n-z}$ is the rightmost end of the removed cap.

\end{corollary}
\begin{remark}\label{rmk:z_interpretation}
 Equivalently, $z$ is the number of caps whose right end is (strictly) to the right of the removed cap. 
\end{remark}
\begin{example}
For the weight $$\lam = \eps_3 + 3\eps_4 + 3\eps_5  + 6\eps_6+ 8\eps_7 + 8\eps_8 + 8\eps_9$$ of $\p(9)$ as described in Example \ref{ex:DS_computation}, the cap diagram will be 

$$ 
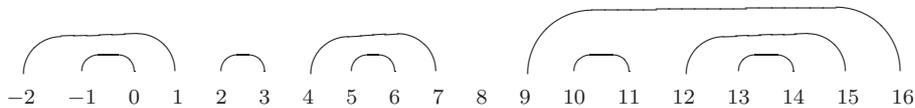
\xymatrix@=5pt{  &{} &{} &{} &{} &{} &{} &{} &{} &{} &{} &{} &{} &{} &{} &{}\\  
 &\underset{-2}{} &\underset{-1}{} 
&\underset{0}{} \ar@{-} `u/6pt[l] `/6pt[l] &\underset{1}{} \ar@{-} `u/15pt[l] `/15pt[lll]  &\underset{2}{} &\underset{3}{}  \ar@{-} `u/6pt[l] `/6pt[l]&\underset{4}{} &\underset{5}{} &\underset{6}{} \ar@{-} `u/6pt[l] `/6pt[l]  &\underset{7}{ } \ar@{-} `u/15pt[l] `/15pt[lll] &\underset{8}{}  &\underset{9}{} &\underset{10}{} &\underset{11}{} \ar@{-} `u/6pt[l] `/6pt[l]  &\underset{12}{} &\underset{13}{}  &\underset{14}{} \ar@{-} `u/6pt[l] `/6pt[l] &\underset{15}{} \ar@{-} `u/15pt[l] `/15pt[lll] &\underset{16}{} \ar@{-} `u/25pt[l] `/25pt[lllllll]  }$$

Then the simple factors of $DS_{x_9}(L_9(\lam))$ are $\Pi L_8(\mu_1), L_8(\mu_2), L_8(\mu_3), L_8(\mu_4)$ as in Example \ref{ex:DS_computation}, and the corresponding cap diagrams are as follows:
\begin{align*}
 \mu_1  \; &\xymatrix@=5pt{  &{} &{} &{} &{} &{} &{} &{} &{} &{} &{} &{} &{} &{} &{} &{}\\  
 &\underset{-2}{} &\underset{-1}{} 
&\underset{0}{} \ar@{-} `u/6pt[l] `/6pt[l] &\underset{1}{}   &\underset{2}{} &\underset{3}{}  \ar@{-} `u/6pt[l] `/6pt[l]&\underset{4}{} &\underset{5}{} &\underset{6}{} \ar@{-} `u/6pt[l] `/6pt[l]  &\underset{7}{ } \ar@{-} `u/15pt[l] `/15pt[lll] &\underset{8}{}  &\underset{9}{} &\underset{10}{} &\underset{11}{} \ar@{-} `u/6pt[l] `/6pt[l]  &\underset{12}{} &\underset{13}{}  &\underset{14}{} \ar@{-} `u/6pt[l] `/6pt[l] &\underset{15}{} \ar@{-} `u/15pt[l] `/15pt[lll] &\underset{16}{} \ar@{-} `u/25pt[l] `/25pt[lllllll]  } \\ 
\mu_2 \; &\xymatrix@=5pt{  &{} &{} &{} &{} &{} &{} &{} &{} &{} &{} &{} &{} &{} &{} &{}\\  
 &\underset{-2}{} &\underset{-1}{} 
&\underset{0}{} \ar@{-} `u/6pt[l] `/6pt[l] &\underset{1}{} \ar@{-} `u/15pt[l] `/15pt[lll]  &\underset{2}{} &\underset{3}{}  &\underset{4}{} &\underset{5}{} &\underset{6}{} \ar@{-} `u/6pt[l] `/6pt[l]  &\underset{7}{ } \ar@{-} `u/15pt[l] `/15pt[lll] &\underset{8}{}  &\underset{9}{} &\underset{10}{} &\underset{11}{} \ar@{-} `u/6pt[l] `/6pt[l]  &\underset{12}{} &\underset{13}{}  &\underset{14}{} \ar@{-} `u/6pt[l] `/6pt[l] &\underset{15}{} \ar@{-} `u/15pt[l] `/15pt[lll] &\underset{16}{} \ar@{-} `u/25pt[l] `/25pt[lllllll]  } \\ 
\mu_3  \;  &\xymatrix@=5pt{  &{} &{} &{} &{} &{} &{} &{} &{} &{} &{} &{} &{} &{} &{} &{}\\  
 &\underset{-2}{} &\underset{-1}{} 
&\underset{0}{} \ar@{-} `u/6pt[l] `/6pt[l] &\underset{1}{} \ar@{-} `u/15pt[l] `/15pt[lll]  &\underset{2}{} &\underset{3}{}  \ar@{-} `u/6pt[l] `/6pt[l]&\underset{4}{} &\underset{5}{} &\underset{6}{} \ar@{-} `u/6pt[l] `/6pt[l]  &\underset{7}{ }  &\underset{8}{}  &\underset{9}{} &\underset{10}{} &\underset{11}{} \ar@{-} `u/6pt[l] `/6pt[l]  &\underset{12}{} &\underset{13}{}  &\underset{14}{} \ar@{-} `u/6pt[l] `/6pt[l] &\underset{15}{} \ar@{-} `u/15pt[l] `/15pt[lll] &\underset{16}{} \ar@{-} `u/25pt[l] `/25pt[lllllll]  } \\ 
\mu_4  \;  &\xymatrix@=5pt{  &{} &{} &{} &{} &{} &{} &{} &{} &{} &{} &{} &{} &{} &{} &{}\\  
 &\underset{-2}{} &\underset{-1}{} 
&\underset{0}{} \ar@{-} `u/6pt[l] `/6pt[l] &\underset{1}{} \ar@{-} `u/15pt[l] `/15pt[lll]  &\underset{2}{} &\underset{3}{}  \ar@{-} `u/6pt[l] `/6pt[l]&\underset{4}{} &\underset{5}{} &\underset{6}{} \ar@{-} `u/6pt[l] `/6pt[l]  &\underset{7}{ } \ar@{-} `u/15pt[l] `/15pt[lll] &\underset{8}{}  &\underset{9}{} &\underset{10}{} &\underset{11}{} \ar@{-} `u/6pt[l] `/6pt[l]  &\underset{12}{} &\underset{13}{}  &\underset{14}{} \ar@{-} `u/6pt[l] `/6pt[l] &\underset{15}{} \ar@{-} `u/15pt[l] `/15pt[lll] &\underset{16}{}.   }
\end{align*}
\end{example}

\subsection{Proof of \texorpdfstring{}{Main } Theorem \texorpdfstring{\ref{thrm:DS_clusters}}{}: auxiliary results, part 1}\label{ssec:DS_aux1}

\mbox{}

{\it Throughout this subsection, we consider all modules in $\F_n$, $\F_{n-1}$ up to parity switch.}

\begin{lemma}\label{lem:DS_cluster_start}
Let $L_n(\lam)$ as above. If $[DS(L_n(\lam)):DS(L_{n-1}(\mu))] \neq 0$ then we have: $f_{\mu}(i-1)=0, f_{\mu}(i)=1$ implies $f_{\lam}(i-1)=0, f_{\lam}(i)=1$.

\end{lemma}
\begin{proof}
  Assume the contrary. Then there exists a position $i$ which is the beginning  of a cluster in $d_{\mu}$ but not in $d_{\lam}$. 
  
  Apply the translation functor $\Theta_i$ to both $L_n(\lam)$ and $L_{n-1}(\mu)$. By \cite[Corollary 8.2.2.]{BDE:periplectic}, the functor $\Theta_i: \F_m \to \F_m$ ($m \geq 1$) annihilates any simple module $L_m(\tau)$ unless $d_{\tau}$ has a black ball in position $i$ and a white ball in position $i-1$. Hence $$\Theta_i(L_n(\lam))=0, \;\; \Theta_i(L_{n-1}(\mu)) \neq 0.$$
 But $\Theta_i$ is an exact functor, so $\Theta_i(L_{n-1}(\mu))$ is a subquotient of $\Theta_i(DS(L_n(\lam))) \cong DS(\Theta_i(L_n(\lam))) =0$. This contradicts our observation that $\Theta_i(L_{n-1}(\mu)) \neq 0$, and the claim of the Lemma follows.

\end{proof}

\begin{proposition}\label{prop:DS_cluster_growth}
 Assume $\left[DS(L_n(\lam)):L_{n-1}(\mu)\right] \neq 0$.
 
 Then for any $i\in \Z$, we have: $f_\lam(i) \geq f_\mu(i)$. That is, if a position in $d_{\lam}$ was empty, so is the corresponding position in $d_{\mu}$.
\end{proposition}
\begin{proof} Define $\mathcal M$ as the set of all quintuples $(\lambda,\mu,i,j,k)$ satisfying the following conditions
  \begin{enumerate}
  \item $[DS_x(L_n(\lambda)):\Pi^zL_{n-1}(\mu)]\neq 0$ \InnaD{for some $z\in \{0,1\}$};
  \item $f_\lambda(j)=0$, $f_\mu(j)=1$ and $j$ is minimal with this property \InnaD{(that is, for any $s<j$ we have: $f_{\lam}(s)\geq f_{\mu}(s)$)};
  \item $i\leq j$ and $f_\mu(i)=f_{\mu}(i+1)=\dots=f_{\mu}(j-1)=1$, $f_\mu(i-1)=0$;
  \item $k$ is the number of $s<j$ such that $f_\mu(s)=1$.
  \end{enumerate}
  
  By Lemma \ref{lem:DS_cluster_start} we have that
  \begin{equation}\label{eq:alpha}
    k\geq 1,\quad i<j,\quad f_{\lambda}(i)=f_{\lambda}(i+1)=\dots=f_{\lambda}(j-1)=1.
\end{equation}
  Our goal is to prove
  that $\mathcal M=\emptyset$. Let us assume that $\mathcal M$ is not empty and let $k$ be minimal with property $(\lambda,\mu,i,j,k)\in \mathcal M$
  for some $\lambda,\mu,i,j$.
  Let $\lambda'$ and $\mu'$ be obtained from $\lambda$ and $\mu$ respectively by moving $\bullet$ from $i$ to $i-1$.
  We are going to prove the following
  \begin{lemma}\label{lem:aux} If $(\lambda,\mu,i,j,k)\in \mathcal M$ then $(\lambda',\mu',i+1,j,k)\in \InnaD{\mathcal M}$.
  \end{lemma}
  \begin{proof} By Lemma \ref{lem:transimple} \eqref{itm:transimple_1} $\Theta_i^{(n-1)}(L_{n-1}(\mu))$ has a composition factor $\Pi^{i+1}(L_{n-1}(\mu'))$. This composition factor appears
    in $DS_x(\Theta^{(n)}_i(L_n(\lambda)))$. Therefore it appears in $DS_x(L_n(\nu))$ for some composition factor $L_n(\nu)$ in $\Theta^{(n)}_i(L_n(\lambda))$.
    We claim that $\nu=\lambda'$. Indeed, by Lemma \ref{lem:DS_cluster_start} we have $f_\nu(i)=0$,  $f_\nu(i+1)=1$ since $f_{\mu'}(i)=0$, $f_{\mu'}(i+1)=1$. 
    If $\nu\neq \InnaD{\lambda'}$, Lemma \ref{lem:transimple} \eqref{itm:transimple_3} implies that $f_\nu(i-1)=0 \InnaD{< f_\mu(i-1) = 1}$. 
    
    \InnaD{Let us show that $i-1$ is the minimal position with such property. Indeed, $f_{\nu}(i-1)=f_{\nu}(i) =0$. Hence by Lemma \ref{lem:transimple} \eqref{itm:transimple_4} we have: $$\forall s \leq i-1, \; f_\lam(s) = f_\nu(s)$$
    
    Furthermore, by our assumption $(\lambda,\mu,i,j,k) \in \mathcal M$, so $$\forall s \leq i-1<j, \; f_{\nu}(s) = f_{\lam}(s)\geq f_{\mu}(s) = f_{\mu'}(s).$$
    
    }
    Hence $(\nu,\mu',i',i-1,k')\in \mathcal M$ for some $i'<i-1$ and  $k'<k$.
    Since $k$ is chosen minimal this is impossible.
  \end{proof}
  Proposition follows from this lemma since after applying it several times we get a tuple of the form $(\lambda'',\mu'',j,j,k)\in \mathcal M$ which is impossible by
  (\ref {eq:alpha}).
  \end{proof}

The next statements will rely on the following corollary of Proposition \ref{prop:DS_cluster_growth}:
\begin{corollary}\label{cor:aux_DS_factors}
 If $[DS(L_n(\lam)):L_{n-1}(\mu)]\neq 0$ then $d_{\mu}$ is obtained from $d_{\lam}$ by removing one black ball from the end of some cluster. 
\end{corollary}

\begin{definition}\label{def:clubsuit}
 Let $\alpha$ be a dominant integral weight for $\p(n)$. Denote by $\alpha^\clubsuit$ the weight whose diagram is obtained from $d_{\alpha}$ by moving each black ball through the longest solid arrow originating at this position.
\end{definition}

\begin{lemma}\label{lem:clubsuit_dual}
 Let $\alpha$ be a dominant integral weight for $\p(n)$. Let $\alpha^*$ be the highest weight of the dual module $L_n(\alpha)^*$. Then $d_{\alpha^*}$ is obtained from $d_{\alpha^\clubsuit}$ by reflecting with respect to position $0$.
\end{lemma}
\begin{proof}
 This is a direct consequence of \cite[Propositions 3.7.1, 8.3.1]{BDE:periplectic}.
\end{proof}

\begin{remark}
 \InnaD{In Proposition \ref{prop:existence}, we also use the weight $\alpha^{\dagger}$, defined in \cite[Section 5.3]{BDE:periplectic}. Its weight diagram $d_{\alpha^{\dagger}}$ is obtained from $d_{\alpha^*}$ by reflecting with respect to the position $(n-1)/2$. Hence $d_{\alpha^{\dagger}}$ is a shift of $d_{\alpha^\clubsuit}$ to the right by $n-1$ positions.} 
\end{remark}

\begin{proposition}\label{prop:DS_nonremovable}
 Assume $[DS(L_n(\lam)):L_{n-1}(\mu)] \neq 0$. Then $d_{\mu}$ satisfies the Initial Segment Condition in Theorem \ref{thrm:DS_clusters}\eqref{itm:DS_clusters2}.
\end{proposition}

\begin{proof}

\InnaC{By Corollary \ref{cor:aux_DS_factors}, $d_{\mu}$ was obtained from $d_{\lam}$ by removing a single black ball.}

 Assume \InnaC{that the statement of the proposition is false: that is,} $[DS(L_n(\lam)):L_{n-1}(\mu)] \neq 0$ and $d_{\mu}$ was obtained from $d_{\lam}$ by removing a black ball in position $i$, where $i$ satisfies: 
 \begin{itemize}
  \item $f_{\lam}(i)=1$, $f_{\lam}(i+1)=0$.
  \item There exists $j\geq i+1$ such that $\InnaC{r_{\lam}(j+1, i+1) >0}$. That is, the segment $[i+1, j]$ contains more black balls than it has empty positions. 
 \end{itemize}

 Consider the minimal $j \geq i+1$ as above. In that case, we must have: 
  \begin{itemize}
  \item $f_\lam(j)=1$, 
  \item $\InnaC{r_{\lam}(j, i+1) =0}$ (that is, the segment $[i+1, j-1]$ contains equal amounts of black balls and empty positions).
  \item \InnaC{$\forall i<k<j, \; r_{\lam}(k+1, i+1) =0$. That is,} the segment $[i+1, k]$ contains no more black balls than it has empty positions. 
 \end{itemize}
 
 From this, we conclude that in the diagram $d_\lam$, there is a solid arrow from $j$ to $i+1$:
 $$d_\lam = \xymatrix{  &\underset{i}{\bullet} &\underset{i+1}{\circ} &{\ldots}  &\underset{j}{\bullet} \ar@/_0.8pc/[luld]}$$
 
 Since $f_{\lam}(i)=1$, we may conclude that this is {\bf not} the longest solid arrow from $j$ to $i+1$ in $d_{\lam}$.
 
 On the other hand, in $d_{\mu}$, we have:
$f_{\mu}(i)=0$, $f_{\mu}(s)=f_{\lam}(s)$ for any $s\neq i$.
  
  Hence in $d_{\mu}$ we also have a solid arrow from $j$ to $i+1$:
    $$d_\mu = \xymatrix{  &\underset{i}{\circ} &\underset{i+1}{\circ} &{\ldots}  &\underset{j}{\bullet} \ar@/_0.8pc/[luld]}$$
  and it is the longest solid arrow from $j$ to $i+1$ in $d_{\lam}$.
  
  We now construct $\lam^\clubsuit$ and $\mu^\clubsuit$. These are obtained by moving each black ball through the longest solid arrow originating at this position. Hence we have:
  $$d_{\lam^\clubsuit} = \xymatrix{  &\underset{i}{\bullet} &\underset{i+1}{\circ} &{\ldots}  &\underset{j}{\circ} } \;\; \text{ and } \;\;d_{\mu^\clubsuit}  = \xymatrix{  &\underset{i}{\circ} &\underset{i+1}{\bullet} &{\ldots}  &\underset{j}{\circ}}$$
  By the Lemma \ref{lem:clubsuit_dual}, we have: 
   $$d_{\lam^*} = \xymatrix{ &\underset{-i-1}{\circ} &\underset{-i}{\bullet} } \;\; \text{ and } \;\;d_{\mu^*}  = \xymatrix{   &\underset{-i-1}{\bullet} &\underset{-i}{\circ}}$$
   
 Hence $f_{\lam^*}(-i-1)=0$, $f_{\mu^*}(-i-1)=1$.
 
 Yet the $DS$ functor commutes with the duality functor (up to isomorphism), so 
 \begin{align*}
&[DS(L_n(\lam^*)):L_{n-1}(\mu^*)]=[DS(L_n(\lam)^*):L_{n-1}(\mu)^*]=\\&=[DS(L_n(\lam))^*:L_{n-1}(\mu)^*]=[DS(L_n(\lam)):L_{n-1}(\mu)] \neq 0
 \end{align*}

 Hence we may apply Proposition \ref{prop:DS_cluster_growth}, and conclude that $$\forall k\in \Z, \, f_{\lam^*}(k) \geq f_{\mu^*}(k).$$ But this contradicts our previous conclusion that $f_{\lam^*}(-i-1)=0$, $f_{\mu^*}(-i-1)=1$.
 
 This completes the proof of the proposition.
\end{proof}

\subsection{Proof of \texorpdfstring{}{Main } Theorem \texorpdfstring{\ref{thrm:DS_clusters}}{}: auxiliary results, part 2}\label{ssec:DS_aux2}

In this subsection we distinguish between simple representations varying by a parity switch. We will also use cap diagrams instead of arrow diagrams, since they suit our needs better in this instance.

\begin{lemma}\label{lem:DS_remove_rightmost}
  If $d_{\mu}$ is obtained from $d_{\lam}$ by removing the rightmost black ball, then $$[DS(L_n(\lam)):L_{n-1}(\mu)] =1.$$
 \end{lemma}

 \begin{proof}
The module $L_n(\lam)$ is a highest weight module with respect to the Borel subalgebra $\fb'=\fb_0 \oplus \p(n)_1 \subset \p(n)$. The roots corresponding to $\p(n)_{-1}$ are $-\eps_i - \eps_j$ for $\eps_i \neq \eps_j$. 
  
Thus we have the following observation: any weight $\alpha$ in $L_n(\lam)$ can be written as $$\alpha = \lam +\sum_{1\leq i \neq j\leq n} s_{ij} (\eps_i + \eps_j)  +\sum_{1\leq i < j\leq n} t_{ij} (\eps_i - \eps_j)$$ for some $s_{ij} \in \{0,1\}$  and $t_{i j} \geq 0$.

Now, we show that $[DS(L_n(\lam)):L_{n-1}(\mu)]\leq  1$. Indeed, given a weight $\alpha$ in $L_n(\lam)$ such that $\alpha_i=\lam_i$ for $i <n$, we necessarily have $\alpha=\lam$ by the observation above. The weight $\lam$ appears in $L_n(\lam)$ with multiplicity $1$, hence $[DS(L_n(\lam)):L_{n-1}(\mu)]\leq  1$.
 
Finally, we show that $[DS(L_n(\lam)):L_{n-1}(\mu)] \neq 0$: 
 Let $v \neq 0$ be the (even) highest weight vector in $L_n(\lam)$ with respect to the Borel subalgebra $\fb$. Then $x.v$ must have weight $\lam + 2\eps_n$, which by the observation above is not a weight of $L_n(\lam)$. Hence $x.v=0$. 
 
 Now, assume that $v\in \Im(x)$. Let us write $v = x.w$ for some $w\in L_n(\lam)$. Then 
$w$ has weight $\lambda -2\eps_n$, which by the reasoning above is impossible. Hence $v \notin \Im(x)$. This implies that $v$ has non-zero (even) image $\tilde{v}$ in $DS(L_n(\lam)) = \Ker(x)/\Im(x)$, and its image has weight $\mu$. 

Now, the vector ${v}$ is a primitive vector with respect to the Borel ${\fb}$, hence the (even) vector $\tilde{v}$ is a primitive vector with respect to the Borel $\tilde{\fb}$, as required.
This completes the proof of the lemma.
\end{proof}

\begin{proposition}\label{prop:existence} Let $d_{\mu}$ be obtained from $d_{\lambda}$ by removing a black ball whose cap is maximal. Then 
  there exists a unique $z \in \Z/2\Z$ such that $[DS(L_n(\lam)):\Pi^z L_{n-1}(\mu)]=1$, moreover $z$ equals the parity of number of balls to the right of
  the removed ball.
\end{proposition}

In order to prepare for the proof of Proposition \ref{prop:existence}, we begin by proving the following.
  \begin{lemma}\label{lem:existence1}  Let $n>1$. Suppose that $d_\lam$ and $d_\mu$ have the \VeraB{leftmost} ball in the same position
      and $d_{\lam'}$ $d_{\mu'}$ are obtained from $d_\lam$ and $d_\mu$ by removing this ball. Then we have
      $$[DS(L_n(\lam)):\Pi^z L_{n-1}(\mu)]=[DS(L_{n-1}(\lam')):\Pi^z L_{n-2}(\mu')]$$
      \InnaC{where $z$ as in Proposition \ref{prop:existence}.}
      \end{lemma}
      \begin{proof} Let $h_1,\dots,h_n$ be the basis in the Cartan subalgebra of $\p(n)$ dual to $\eps_1,\dots\eps_n$.
        We have a decomposition
        $$L_n(\lam)=\bigoplus_{i\geq \lambda_1}L_n(\lam)^i$$
        where $L_n(\lam)^i$ is the eigenspace of $h_1$ with eigenvalue $i$. Every component $L_n(\lambda)^i$ is a module over the centralizer $\mathfrak l$
        of $h_1$. Since $x\in\mathfrak l$ we have
        $$DS(L_n(\lam))=\bigoplus_{i\geq \lambda_1}DS(L_n(\lam)^i).$$
        Note that $\mathfrak l$ is the direct sum $\mathbb Ch_1\oplus \mathfrak l'$ where $\mathfrak l'$ is another copy of $\p(n-1)$ inside $\p(n)$.
        Furthermore, $L_n(\lam)^{\lambda_1}$ is isomorphic $L_{n-1}(\lam')$ since $L_n(\lam)$ is a quotient of the parabolically induced module
        $U(\p(n))\otimes _{U(\mathfrak{b}+\mathfrak{l})}L_n(\lam)^{\lambda_1}$. Now it is clear that if $\mu_1=\lambda_1$ then $L_{n-1}(\mu)$ occurs
        in $DS(L_n(\lam))$ with the same multiplicity as $L_{n-1}(\mu)^{\lambda_1}$ occurs in $DS (L_n(\lam)^{\lambda_1})$. The statement follows.
       \end{proof}

       Consider the Borel subalgebra ${\mathfrak{b}}_n^\dagger$ of $\p(n)$ with simple roots \VeraB{ $2\eps_1,\eps_2-\eps_1,\dots,\eps_n-\eps_{n-1}$}, \InnaB{and the corresponding Borel subalgebra ${\mathfrak{b}}_{n-1}^\dagger$ of $\p(n-1)$ with simple roots $2\eps_1,\eps_2-\eps_1,\dots,\eps_{n-1}-\eps_{n-2}$}.
       Let ${\lam}^\dagger$ denote the highest weight
       of $L_n(\lambda)$ with respect to ${\mathfrak{b}}_n^\dagger$, and similarly for weights of $\p(n-1)$. We will denote by $L_n^\dagger(\nu)$ the simple $\p(n)$-module of highest weight $\nu$ respect to
       ${\mathfrak{b}}_n^\dagger$ having an even highest weight vector, and similarly for simple $\p(n-1)$-modules.

       One readily sees that
  $$L_{n-1}(\mu)\simeq\Pi^sL^\dagger_{n-1}(\mu^{\InnaB{\dagger}}), \quad L_n(\lambda)\simeq \Pi^tL^\dagger_{n}(\lambda^\dagger) $$ where
  \begin{equation}\label{eq:existence_s_t}
s=\sum_{i=1}^{n-1} \mu^\dagger_i-\mu_i,\quad t=\sum_{i=1}^n \lambda^\dagger_i-\lambda_i.   
  \end{equation}

       Let $y\in\p(n)$ be a root vector of weight $2\varepsilon_1$. Then by the same argument as in
       the proof Lemma \ref{lem:DS_remove_rightmost}, we have:
       \begin{lemma}\label{lem:DS_remove_lefttmost}
  Let $d_{\nu}$ be obtained from $d_{\lam^{\InnaB{\dagger}}}$ by removing the leftmost black ball and shifting all other black balls one position left, then $[DS_y(L^\dagger_n(\lam^{\InnaB{\dagger}})):L_{n-1}^\dagger(\nu)] =1$.
\end{lemma}
\begin{remark} The shift is necessary due to renumeration $2\mapsto 1,\dots,n\mapsto n-1$.
  \end{remark}

 A combinatorial algorithm of computing ${\lam}^\dagger$ in terms of weight diagrams is given in \cite{BDE:periplectic}.
 Enumerate the balls from left to right. Let $1\leq a\leq b\leq n$. Define the operation $D_{a,b}$ on the set of diagrams as follows:
if  the positions next right to both  $a$-th and $b$-th balls in a diagram $d$ are free, then $D_{a,b}(d)$ is obtained by moving both balls one position right.
 Otherwise $D_{a,b}(d)=d$. Then
\VeraB{ $$d_{\lambda^\dagger}=D_{1,2}\dots D_{1,n}D_{2,3}\dots D_{2,n}\dots D_{n-2,n-1}D_{n-2,n} D_{n-1,n}(d_{\lambda}).$$}

\begin{definition}
We will say that a cap $c = (i,j), i<j$ {\it covers} a black ball in a given weight diagram $d_{\lam}$ if the position $k$ of the black ball satisfies: $i<k<j$.
\end{definition}

 \begin{lemma}\label{lem:existence2} We have $\bar\lambda^{\dagger}_1-\bar\lambda_1=n-m_1-1$ where $m_i$ is the number of caps which cover the $i$-th black ball in $d_{\lam}$. In particular,
   if the cap ending at the first black ball is maximal then $\bar\lambda^{\dagger}_1-\bar\lambda_1=n-1$.
 \end{lemma}
 
 \begin{proof} One proves the statement by induction on $n$. 
 \InnaC{{\bf Base:} let $n=1$. Then $m_1=0$ and $\bar\lambda^{\dagger}_1-\bar\lambda_1$ as required.
 
 {\bf Step:} Let $n>1$ and assume the statement holds for $n-1$.}
 
 Let $\alpha \in \Lambda_n$ be the weight defined by
   $$d_{\alpha}:=D_{2,3}\dots D_{2,n}\dots D_{n-1,n}(d_{\lambda})$$
\InnaC{and let $\lambda', \alpha' \in \Lambda_{n-1}$ be the weights whose diagrams $d_{\lambda'}, d_{\alpha'}$  are obtained from $d_{\lam}, d_{\alpha}$ respectively by removing the leftmost black ball in each diagram. Then $\alpha' = \lambda'^{\dagger}$, so by the induction assumption, we have: $$\bar{\alpha'}_1 - \bar{\lambda'}_1 =  \bar\alpha_2-\bar\lambda_2=n-2-m_2.$$}

   Now, consider first the case when $m_1>0$. Then $m_2>0$ and  $\bar\lambda_2-\bar\lambda_1=m_2-m_1+\VeraB{2}$.
   Recall that we have:  
   $\bar\alpha_2-\bar\lambda_2=n-2-m_2$ and hence $\bar\alpha_2-\bar\lambda_1=\InnaB{n-m_1}$. Using
   $d_{\lambda^\dagger}=D_{1,2}\dots D_{1,n}(d_{\alpha})$ we get that we can move the first ball until it stays next to the second ball of $d_\alpha$, namely exactly
   $n-1-m_1$ times. Hence $\bar\lambda^{\dagger}_1-\bar\lambda_1=n-m_1 \InnaB{-1}$.

   Now let $m_1=0$. Then $\bar\lambda_2-\bar\lambda_1\geq m_2+1$. Recall that we have: $\bar\alpha_2-\bar\lambda_2=n-2-m_2$ and hence
   $\bar\alpha_2-\bar\lambda_1\geq n-1$. Hence we move the first ball $n-1$ times.
  \end{proof}
  Now we are ready to prove Proposition \ref{prop:existence}:
  
  \begin{proof}[Proof of Proposition \ref{prop:existence}]
Note that the fact that a ball is the end of a maximal cap depends only on positions of the black balls to its right. Therefore Lemma~\ref{lem:existence1}
       implies that it suffices to prove the statement of Proposition \ref{prop:existence} in the case when the removed ball is the leftmost black ball in the diagram $d_\lambda$. 
       
       Assume $d_\mu$ is of this form: namely, $d_{\mu}$ is obtained from $d_\lambda$ by removing the leftmost black ball (from position $\lambda_1$). Since $\lam, \mu$ should satisfy the condition of Proposition \ref{prop:existence}, the cap ending in position $\lambda_1$ is maximal, hence $m_1=0$ in the notation of Lemma \ref{lem:existence2}.
       
       Let $d_\nu$ be the diagram obtained from $d_{\lam^\dagger}$ as in Lemma \ref {lem:DS_remove_lefttmost}.
  Then we have \InnaB{$\nu = \mu^{\dagger}$} and $$[DS_yL_n^{\dagger}(\lambda^{\InnaB{\dagger}}):L_{n-1}^\dagger(\nu)] = [DS_yL_n^{\dagger}(\lambda^{\InnaB{\dagger}}):L_{n-1}^\dagger(\mu^{\dagger})]=1.$$ 
  Note that $DS_y$ and $DS = DS_x$ are isomorphic functors since $y$ and $x$ are conjugate by the adjoint action of $GL(n)$. Let $t,s $ as in \eqref{eq:existence_s_t}. We obtain:
$$ \InnaB{[DS L_n(\lambda):L_{n-1}(\mu)] =  [\Pi^{t} DS_yL_n^{\dagger}(\lambda^{\dagger}): \Pi^{s}L_{n-1}^\dagger(\mu^{\dagger})]}.$$ Observing that $t-s=n-1$ gives us the required statement.
 
  \end{proof}

\subsection{Action of the \texorpdfstring{$DS$}{Duflo-Serganova} functor: corollaries}\label{ssec:corollaries_DS}

Let $x_n\in \p(n)_1$, and $DS=DS_{x_n}$ as before. The following are direct corollaries of Theorem \ref{thrm:DS_clusters}:
\begin{corollary}\label{cor:comp_factors_DS}
 Let $\lam \in \Lambda_n$. The number of composition factors of $DS(L_n(\lam))$ is precisely the number of maximal arrows (or maximal caps).
\end{corollary}

\begin{corollary}
 Let $\lam \in \Lambda_n$. Then $DS(L_n(\lam))$ is simple iff there exists exactly one maximal solid arrow (one maximal cap) in $d_{\lam}$.
\end{corollary}

\section{Computation of superdimensions}\label{sec:sdim}
In this section we compute the superdimension of the simple $\p(n)$-modules in $\F_n$. 

\subsection{Forests}\label{ssec:forest}
Let $\lam\in \Lambda_n$ be a dominant integral weight, and let $d_{\lam}$ be its weight diagram with caps. Let $(C(\lam), \lhd)$ be the poset of caps in $d_{\lam}$ with partial order $\lhd$ described in Definition \ref{def:caps_poset}.

We define an augmented poset $$(\widehat{C}(\lam),\,\lhd), \;\; \widehat{C}(\lam) = C(\lam) \sqcup \{c_{*}\}$$ where $c_{*}$ is a {\it ``virtual cap''} which is defined to be the greatest element in $\widehat{C}(\lam)$: namely, we have $$c_{*} \notin C(\lam), \; \text{ and } \;\forall c \in C(\lam), \, c \lhd c_*.$$ 

We define the successors of $c_*$ as in Definition \ref{def:caps_poset}. These are precisely the maximal caps in $C(\lam)$.

\begin{definition}\label{def:caps_prop}
 \mbox{}
 \begin{itemize}
  \item Given a cap $c \in \widehat{C}(\lam)$, let $$int(c) = \sharp \{c' \in \widehat{C}(\lam): c \rhd c'\}$$ be the number of caps internal to $c$, including $c$ itself. 
  
  If $c=(i, j)$ is a non-virtual cap, then $int(c)$ is the number of black balls in $d_{\lam}$ between positions $i$ and $j$ (including position $j$), and $int(c_*)=n+1$.
  
  \item A cap $c \in \widehat{C}(\lam)$ with $int(c) \equiv 0 \mod 2$ is called an {\it even} cap; otherwise it is called an {\it odd} cap.
  
  \item If every cap $c \in \widehat{C}(\lam)$ has at most one odd successor, we call such a weight $\lam$ {\it worthy}.  
 \end{itemize}

\end{definition}
\begin{example}\label{ex:worthy_wt1}

Consider the weight $\lambda = \eps_1 + \eps_2 + 3\eps_3+5\eps_4+5\eps_5+5\eps_6$ for $\p(6)$ \InnaC{as in Examples \ref{ex:wt1}, \ref{ex:wt1_caps}}. The cap diagram for $\lambda$ is:  

$$ \xymatrix@=20pt{  &{} &{} &{} &{} &{} &{} &{} &{} &{} &{} &{} &{} &{} &{} &{}\\  &\underset{-1}{} &\underset{0}{} 
&\underset{1}{} \ar@{-} `u/10pt[l] `/10pt[l] &\underset{2}{} \ar@{-} `u/20pt[l] `/20pt[lll] &\underset{3}{} &\underset{4}{} &\underset{5}{} \ar@{-} `u/10pt[l] `/10pt[l] &\underset{6}{} &\underset{7}{ } &\underset{8}{}  \ar@{-} `u/10pt[l] `/10pt[l]  &\underset{9}{} \ar@{-} `u/20pt[l] `/20pt[lll] &\underset{10}{} \ar@{-} `u/40pt[l] `/40pt[lllllll] }$$

Here $c_*$ has two successors: $(-1,2), (3,10)$ (both even caps), and we have: 
\begin{align*}
& int(c_*)=2, \; int(\InnaC{(0,1)})=int((4,5))=int((7,8))=1, \\ & int((\InnaC{-1,2}))=int((6,9))=2, \; int((3,10))=4.                                                              \end{align*}
 The odd caps here are $(0,1), (4,5), (7,8)$, and the rest are even. In this case, each cap in $\widehat{C}(\lam)$ has at most one odd successor, so the weight $\lam$ is worthy.
\end{example}

\begin{example}\label{ex:not_worthy_wt}

Consider the weight $\lambda = \eps_1+ 4\eps_2+6\eps_3+6\eps_4$ for $\p(4)$. The cap diagram for $\lambda$ is:  
$$ 
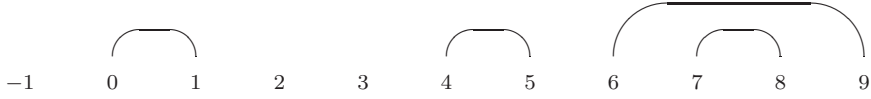
\xymatrix@=20pt{  &{} &{} &{} &{} &{} &{} &{} &{} &{} &{} &{} &{} &{} &{} &{}\\  &\underset{-1}{} &\underset{0}{} 
&\underset{1}{} \ar@{-} `u/10pt[l] `/10pt[l] &\underset{2}{}  &\underset{3}{} &\underset{4}{} &\underset{5}{} \ar@{-} `u/10pt[l] `/10pt[l] &\underset{6}{} &\underset{7}{ } &\underset{8}{} \ar@{-} `u/10pt[l] `/10pt[l]  &\underset{9}{} \ar@{-} `u/20pt[l] `/20pt[lll]  }$$

Here $$int((0,1))=int((4,5))=int((7,8))=1, \;int((6,9))=2.$$ The odd caps here are $(0,1), (4,5), (7,8)$, and the $(6,9)$ is an even cap. The maximal (non-virtual) caps in $C(\lam)$ are $(0,1), (4,5), (6,9)$. Hence the virtual cap has two odd successors, and the weight $\lam$ is not worthy.
\end{example}
\begin{example}\label{ex:worthy_wt2}

Consider the weight $$\lambda = -7\eps_1 -7\eps_2-7\eps_3-5\eps_4-3\eps_5-3\eps_6-\eps_7+\eps_8+\eps_9+\eps_{10}+\eps_{11}$$ for $\p(11)$. The cap diagram for $\lambda$ is:  

$$ 
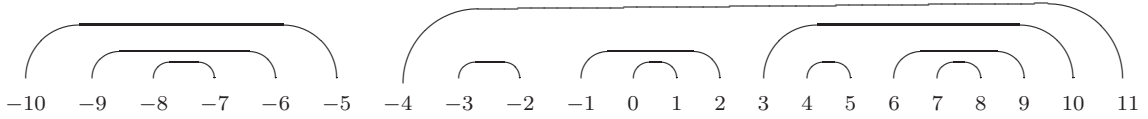
\xymatrix@=5pt{  &{} &{} &{} &{} &{} &{} &{} &{} &{} &{} &{} &{} &{} &{} &{}\\   &\underset{-10}{}&\underset{-9}{}&\underset{-8}{}&\underset{-7}{} \ar@{-} `u/6pt[l] `/6pt[l] &\underset{-6}{} \ar@{-} `u/10pt[l] `/10pt[lll]
&\underset{-5}{} \ar@{-} `u/20pt[l] `/20pt[lllll] &\underset{-4}{}  &\underset{-3}{} 
&\underset{-2}{} \ar@{-} `u/6pt[l] `/6pt[l] &\underset{-1}{} &\underset{0}{} 
&\underset{1}{} \ar@{-} `u/6pt[l] `/6pt[l] &\underset{2}{} \ar@{-} `u/10pt[l] `/10pt[lll] &\underset{3}{} &\underset{4}{} &\underset{5}{} \ar@{-} `u/6pt[l] `/6pt[l] &\underset{6}{} &\underset{7}{ } &\underset{8}{} \ar@{-} `u/6pt[l] `/6pt[l]  &\underset{9}{}\ar@{-} `u/10pt[l] `/10pt[lll] &\underset{10}{} \ar@{-} `u/20pt[l] `/20pt[lllllll] &\underset{11}{} \ar@{-} `u/30pt[l] `/30pt[lllllllllllllll]  }$$

In this case, each cap in $\widehat{C}(\lam)$ has at most one odd successor, so the weight $\lam$ is worthy.
\end{example}
\begin{example}
The zero weight $\lam=0$ is always worthy (for any $n\geq 1$), since it gives a linear order on the augmented set of its caps $\widehat{C}(\lam)$. 
\end{example}
\begin{example}
 The weight $\lambda = -\eps_1$ is not worthy for any $n\geq 2$. For example, for $n=5$, the cap diagram of $\lam$ is 
 
 $$ 
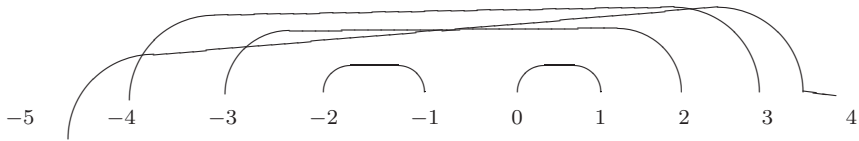
\xymatrix@=20pt{  &{} &{} &{} &{} &{} &{} &{} &{} &{} &{} &{} &{} &{} &{} &{}\\  &\underset{-5}{} &\underset{-4}{} &\underset{-3}{} &\underset{-2}{} &\underset{-1}{} \ar@{-}  `u/10pt[l] `/10pt[l]
&\underset{0}{}  &\underset{1}{}  \ar@{-} `u/10pt[l] `/10pt[l] &\underset{2}{} \ar@{-} `u/25pt[l] `/25pt[lllll]  &\underset{3}{} \ar@{-} `u/35pt[l] `/35pt[lllllll]  &\underset{4}{} \ar@{-} `u/50pt[lllllllll] `/50pt[lllllllll] }$$ 
The cap $(-3, 2)$ has two odd successors, hence $\lambda$ is not worthy.
\end{example}

\begin{remark}
 The virtual cap $c_*$ is even iff $n \equiv 1 \mod 2$.
\end{remark}

The following lemma is straightforward:
\begin{lemma}\label{lem:odd_even_caps}
 Given any weight $\lam \in \Lambda_n$, any even cap has an odd number of odd successors, and any odd cap has an even number of odd successors. 
\end{lemma}

This immediately leads to the following conclusion:
\begin{corollary}\label{cor:odd_even_caps}
  Given a worthy weight $\lam \in \Lambda_n$, we have:
\begin{enumerate}
 \item Given any odd cap, all its successors are even caps.
 \item Given any even cap, it has exactly one odd successor.
\end{enumerate}
\end{corollary}

\begin{definition}\label{def:forest_lambda}
 Let $\lam$ be a worthy weight. We construct a rooted forest ${F}_{\lam}$ as follows.

 \begin{itemize}

  \item The nodes of ${F}_{\lam}$ are pairs $(c_0, c_1)$, where $c_0, c_1 \in \widehat{C}(\lam)$, $int(c_0) \equiv 0 \mod 2$, $int(c_1) \equiv 1 \mod 2$, and $c_1$ is the unique odd successor of $c_0$.
 
  \item There is an edge from a node $v=(c_0, c_1)$ to a node $v'=(c'_0, c'_1)$ in ${F}_{\lam}$ if $c'_0$ is a successor of either $c_0$ or $c_1$.
  
 In that case, we consider the node $v$ a parent of the node $v'$ in our rooted forest.
 \end{itemize}
  The forest ${F}_{\lam}$ is called {\it the rooted forest corresponding to $\lam$}.
\end{definition}

\begin{example}
\mbox{}
 \begin{enumerate}
  \item For $\lam = 0$, $F_{\lam}$ is a linear rooted tree with $\lfloor \frac{n+1}{2} \rfloor$ nodes.
    \item For $\lam$ as in Example \ref{ex:worthy_wt1}, the rooted forest will be $$ \xymatrix{  &{\bullet}    &{}  &{\bullet}  \ar[d]  \\ &{}  &{}  &{\bullet}  }$$
  \item For $\lam$ as in Example \ref{ex:worthy_wt2}, the rooted forest will be $$ \xymatrix{ &{} &{\bullet} \ar[ld] \ar[rd]  &{} &{} \\ &{\bullet}    &{}  &{\bullet}  \ar[ld] \ar[rd] &{} \\ &{}  &{\bullet} &{} &{\bullet} \ar[d]\\ &{}  &{} &{} &{\bullet} }$$
 \end{enumerate}
\end{example}

We also recall the following definitions (cf. \cite{HeiWei14}):
\begin{definition}\label{def:forest factorial}
 Let $F$ be a rooted forest. 
 \begin{itemize}
  \item We denote by $\abs{F}$ the number of nodes in the forest.
  \item For any node $v$ in $F$, we denote by $F^{(v)}$ the rooted subtree of $F$ whose root is $v$.
    \item For any root $v$ in $F$ (that is, $v$ has no parent), we denote by $F\setminus\{v\}$ the rooted forest obtained from $F$ by removing $v$ and all the edges originating in it.
  \item We define the {\it forest factorial} $F!$ by $$F! = \prod_v \abs{F^{(v)}}$$
 \end{itemize}

\end{definition}

\begin{remark}
  Given a worthy weight $\lam \in \Lambda_n$, $\abs{F_{\lam}}=\lfloor {\frac{n+1}{2}}\rfloor$.
\end{remark}

\begin{example}
\mbox{}
 \begin{enumerate}
  \item For $\lam = 0$, we have: $$F_{\lam}! = \lfloor \frac{n+1}{2} \rfloor!$$ 
    \item For $\lam$ as in Example \ref{ex:worthy_wt1}, we have $$F_{\lam}!=1\cdot 2 \cdot 1 = 2, \;  \abs{F_{\lam}} = 3.$$
  \item For $\lam$ as in Example \ref{ex:worthy_wt2}, we have $$F_{\lam}!=6\cdot 1 \cdot 4 \cdot 1 \cdot 2 \cdot 1 = 48, \;  \abs{F_{\lam}} = 6. $$
 \end{enumerate}
\end{example}

The following statements will be useful for Theorem \ref{thrm:sdim}:
\begin{lemma}\label{lem:forest_heap}
  The integer $\frac{\abs{F}!}{F!}$ counts the number of heap-orderings on the rooted forest $F$. Here a {\it heap-ordering} on a rooted forest is a bijection $$\alpha:  \{\text{ nodes of } \, F_\lam \,\} \longrightarrow \{1,2,3,\ldots, \abs{F}\}$$ such that $\alpha(v) \leq \alpha(v')$ whenever $v$ is an ancestor of $v'$ (equivalently, on any subtree, the number corresponding to the root is less or equal to the numbers corresponding the rest of the nodes in that subtree).
\end{lemma}

\begin{proof}
 We prove the statement by (complete) induction on $\abs{F}$. 
 
 {\bf Base}: if $\abs{F}=0$ then the statement is clearly true.
 
 {\bf Step}: let $F$ be a rooted forest with at least $1$ node, and assume the statement holds for any rooted forest with fewer nodes.
 
 Let $v_1, \ldots, v_m$ be the roots of $F$, and let $T_i := F^{(v_i)}$ be the subtree whose root is $v_i$. Then 
\begin{align*}
\frac{\abs{F}!}{F!} &= \frac{\abs{F}!}{\prod_{i=1}^m \abs{T_i}} \cdot \frac{\prod_{i=1}^m \abs{T_i}}{F!} =\binom{\abs{F}!}{\abs{T_1}, \abs{T_2} ,\ldots, \abs{T_m} } \cdot \prod_{i=1}^m \frac{\abs{T_i}!}{T_i!} =\\&= \binom{\abs{F}!}{\abs{T_1}, \abs{T_2} ,\ldots, \abs{T_m} } \cdot \prod_{i=1}^m \frac{\abs{T_i\setminus \{v_i\}}}{\left(T_i\setminus \{v_i\}\right)!}
\end{align*}
The multinomial coefficient $\binom{\abs{F}!}{\abs{T_1}, \abs{T_2} ,\ldots, \abs{T_m} }$ counts the number of ways to partition the set $\{1,2,3,\ldots, \abs{F}\}$ into an ordered multiset of unordered subsets, whose sizes are $\abs{T_1}, \abs{T_2} ,\ldots, \abs{T_m}$. Each such subset will be the set of numbers corresponding to the rooted tree $T_i$, with the smallest number corresponding to the root $v_i$ of $T_i$. 

By the induction assumption, for each $i$ we have: the value $\frac{\abs{T_i\setminus \{v_i\}}}{\left(T_i\setminus \{v_i\}\right)!}$ counts the number of heap-orderings on the rooted forest $T_i\setminus \{v_i\}$, which implies the statement of the lemma. 
\end{proof}
From Lemma \ref{lem:forest_heap} we immediately obtain:
\begin{corollary}\label{cor:forest_binom}
 Given a rooted tree $F$, we have the following identity: $$\frac{\abs{F}!}{F!} = \sum_{v \text{ a root of } F} \frac{\abs{F\setminus\{v\}}!}{\left( F\setminus\{v\}\right)!}$$
\end{corollary}

\subsection{Computation of superdimensions}\label{ssec:sdim_computation}

\begin{theorem}\label{thrm:sdim}
 Let $\lam \in \Lambda_n$ and let $L_n(\lam)$ be the corresponding simple module in $\F_n$ \InnaC{(with an even highest weight vector, as before)}. 
 
 Consider the cap diagram $d_{\lam}$, as described in Section \ref{sssec:cap_diag}.
 
 If the weight $\lam$ is not worthy (see Definition \ref{def:caps_prop}), then $$\sdim L_n(\lam)=0.$$
 
 If the weight $\lam$ is worthy, let $F_{\lam}$ be the corresponding rooted forest (as in Definition \ref{def:forest_lambda} above). Then $$\sdim L_n(\lam) = \frac{\abs{F_{\lam}}!}{F_{\lam}!}.$$
\end{theorem}

\begin{example}
\mbox{}
 \begin{enumerate}
  \item For $\lam = 0$ and any $n\geq 1$, we have: $\sdim L_n(0)= \sdim \triv = \frac{\abs{F_{\lam}}!}{F_{\lam}!}=1$.
   \item For $\lam = -\eps_1$ and $n\geq 2$, we have: $\sdim L_n(-\eps_1)= - \sdim V_n = 0$.
    \item For $\lam$ as in Example \ref{ex:worthy_wt1}, we have: $\sdim L_6(\lam)=\frac{\abs{F_{\lam}}!}{F_{\lam}!}=3$.
    \item For $\lam$ as in Example \ref{ex:not_worthy_wt}, we have: $\sdim L_4(\lam)=0$.
  \item For $\lam$ as in Example \ref{ex:worthy_wt2}, we have: $\sdim L_{11}(\lam)=\frac{\abs{F_{\lam}}!}{F_{\lam}!}=15$.
 \end{enumerate}
\end{example}

\begin{proof}[Proof of Theorem \ref{thrm:sdim}]
We prove the required statement by induction on $n \geq 1$, done separately for odd and even $n$. 

{\bf Base}: 
For $n=1$, any (dominant) integral $\p(1)$-weight $\lam\in \Lam_1$ has a cap diagram with a single cap. So it is worthy, and its rooted forest (tree) $F_\lam$ consists of just one node. The simple $\p(1)$-module $L_1(\lam)$ has superdimension $1$. Hence $$\frac{\abs{F_{\lam}}!}{F_{\lam}!} = 1 = \sdim L_1(\lam)$$ as required.

For $n=2$, we have two types of (dominant) integral $\p(2)$-weights $\lam\in \Lam_1$: 
\begin{enumerate}
 \item If $\lam_1=\lam_2$, then the cap diagram has exactly two caps, one internal to the other:
  $$ 
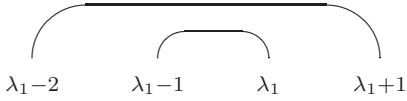
\xymatrix@=20pt{  &{} &{} &{} &{} &{} &{} &{} &{} &{} &{} &{} &{} &{} &{} &{}\\   &\underset{\lam_1-2}{} &\underset{\lam_1-1}{}  &\underset{\lam_1}{} \ar@{-} `u/10pt[l] `/10pt[l]
&\underset{\lam_1+1}{} \ar@{-} `u/20pt[lll] `/20pt[lll] }$$
 So $\lam$ is worthy. Its rooted forest (tree) $F_\lam$ consists of just one node. The simple $\p(2)$-module $L_2(\lam)$ is a tensor power of the determinant representation of $\p(2)_0 = \gl_2$, and has superdimension $1$. Hence $$\frac{\abs{F_{\lam}}!}{F_{\lam}!} = 1 = \sdim L_2(\lam)$$ as required.
\item If $\lam_1\neq\lam_2$, then the cap diagram has exactly two disjoint caps:
  $$ 
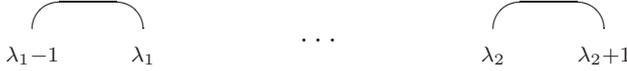
\xymatrix@=20pt{  &{} &{} &{} &{} &{} &{} &{} &{} &{} &{} &{} &{} &{} &{} &{}\\    &\underset{\lam_1-1}{}  &\underset{\lam_1}{} \ar@{-} `u/10pt[l] `/10pt[l]
 &{} &{\ldots} &{} &\underset{\lam_2}{}  &\underset{\lam_2+1}{} \ar@{-} `u/10pt[l] `/10pt[l] }$$
 The virtual cap in this case has two odd successors, hence $\lam$ is not worthy. The simple $\p(2)$-module $L_2(\lam)$ is typical and has superdimension $0$, as required.
\end{enumerate}

{\bf Step}: Assume the statement of the theorem holds for $n-2, n-1$. We now prove it for $n$.

Recall that the Duflo-Serganova functor $DS_x$ (for any $x\in \p(n)_{\bar{1}}$ is a symmetric monoidal functor, so it preserves categorical dimensions (in other words, superdimensions).

For each $k=n-1, n$, let $x_k\in \p(k)_1, x_k\neq 0$ be the odd element corresponding to the root $2\eps_k$. Let $DS_{x_{n-1}}$, $DS_{x_n}$ be the corresponding Duflo-Serganova functors.

First we consider the case when $n \equiv 1 \mod 2$. 

Let $\lam \in \Lambda_n$. Then 
\begin{equation}\label{eq:sdim_aux1}
\sdim L_n(\lam) = \sdim DS_{x_n}(L_n(\lam)) = \sum_{c \in C(\lam) \,\text{ maximal}} (-1)^{z(\lam, c)} \sdim L_{n-1}(\mu_c) 
\end{equation}

Here for each maximal (non-virtual) cap $c$ in $C(\lam)$, we denote by $\mu_c$ the weight in $\Lambda_{n-1}$ such that $d_{\mu_c}$ is obtained from $d_\lam$ by removing the cap $c$ (see Corollary \ref{cor:DS_caps}), and $z(\lam, c) = z$ is the parity of the composition factor $L_{n-1}(\mu_c)$ in $DS_{x_n}(L_n(\lam))$.

Consider a maximal cap $c \in C(\lam)$ as above, and let $\mu := \mu_c$. Then $\widehat{C}(\mu) = \widehat{C}(\lam) \setminus \{c\}$ with induced partial order.

We then have the following sublemma:
\begin{sublemma}\label{sublem:1}
Assume $n \equiv 1 \mod 2$. Then we have:

\begin{itemize}
 \item If $\lam$ was not worthy, then so is $\mu$.
 
 \item If $\lam$ was worthy, and $c$ was even, then $\mu$ will not be worthy. 
\item If $\lam$ was worthy, and $c$ was odd, then $\mu$ will be worthy. 
\end{itemize}
\end{sublemma}
\begin{proof}[Proof of Sublemma]

\begin{itemize}
 \item Assume $\lam$ was not worthy.
 
 Let $c' \in \widehat{C}(\lam)$ be a cap with at least $2$ odd successors. Then we have three cases: 
 \begin{enumerate}
  \item Case $c'=c$. In this case $c_{*} \in \widehat{C}(\mu)$ will have at least $2$ odd successors.
    \item Case $c' = c_{*}$. Recall that since $n\equiv 1 \mod 2$, the virtual cap $c_* \in \widehat{C}(\lam)$ is even, hence it has an odd number of odd successors, by Lemma \ref{lem:odd_even_caps}. Thus it has at least $3$ odd successors in $\widehat{C}(\lam)$, and $c_*\in \widehat{C}(\mu)$ will still have at least $2$ odd successors in $\widehat{C}(\mu)$. 
  \item Case $c' \neq c, c_{*}$. In this case $c'\in \widehat{C}(\mu)$ will have at least $2$ odd successors. 
 \end{enumerate}
 
 In all these cases $\mu$ is not worthy.
 
 \item Assume $\lam$ was worthy, and $c$ was even. 
 
 Since $n\equiv 1 \mod 2$, the virtual cap $c_* \in \widehat{C}(\lam)$ is even. So $c_*$ has one odd successor in $\widehat{C}(\lam)$ which is not $c$, and will gain one more odd successor (a former successor of $c$) after $c$ is removed. Thus $c_*\in \widehat{C}(\mu)$ will still have at least $2$ odd successors, and $\mu$ is not worthy. 
 
 \item Assume $\lam$ was worthy, and $c$ was odd. Then by Corollary \ref{cor:odd_even_caps} the number of odd successors of any given cap has not grown, and hence $\mu$ is worthy.
\end{itemize}
 The sublemma is proved.
\end{proof}

Thus in case $n \equiv 1 \mod 2$, we have: if $\lam$ is not worthy then $\sdim L_n(\lam) = 0$; if $\lam$ is worthy then $$\sdim DS_{x_n}(L_n(\lam)) = (-1)^{z(\lam, c)} \sdim L_{n-1}(\mu)$$ where $\mu \in \Lambda_{n-1}$ is the weight whose cap diagram $d_{\mu}$ is obtained by removing the unique (non-virtual) {\it odd} maximal cap $c$ in $d_{\lam}$. 

This implies that the rooted forest $F_{\mu}$ is obtained from the rooted tree $F_{\lam}$ by removing its root, hence $$\frac{\abs{F_{\mu}}!}{F_{\mu}!} = \frac{\abs{F_{\lam}}!}{F_{\lam}!}.$$

The parity $z(\lam, c)$ appearing in Corollary \ref{cor:DS_caps} is $0$: indeed, since $c$ was the only odd cap in $d_{\lam}$, there is an even number of caps whose right end is to the right of $c$, hence $z(\lam, c)=0$ by Remark \ref{rmk:z_interpretation}.

Applying the induction assumption to $L_{n-1}(\mu)$, we obtain: $$\sdim L_n(\lam) = \sdim DS_{x_n}(L_n(\lam)) = \sdim L_{n-1}(\mu)=\frac{\abs{F_{\mu}}!}{F_{\mu}!} = \frac{\abs{F_{\lam}}!}{F_{\lam}!}$$ as required. This completes the proof of the theorem in case $n$ is odd.

\mbox{}

We now consider the case when $n$ is even.

Again, let $\lam \in \Lambda_n$.

We consider the functor 
$$\overline{DS}: \F_n \to \F_{n-2}, \;\; \overline{DS} := DS_{x_{n-1}} \circ DS_{x_n}$$

Then $\overline{DS}$ is a symmetric monoidal functor preserving superdimensions. 

Computing the action of $\overline{DS}$ on $L_n(\lam)$ explicitly, we have: 
\begin{equation}\label{eq:sdim_aux2}
\sdim L_n(\lam) = \sdim \overline{DS}(L_n(\lam)) = \sum_{\underline{c}=(c_1, c_2), \; c_1, c_2 \in C(\lam)} (-1)^{\tilde{z}(\lam, \underline{c})} \sdim L_{n-2}(\mu_{\underline{c}}) 
\end{equation}

Here the sum goes over all ordered pairs of caps $\underline{c}=(c_1, c_2)$ where $c_1$ is a maximal (non-virtual) cap in $C(\lam)$, while $c_2 \in C(\lam)$ is a successor of either $c_*$ or $c_1$. The weight $\mu_{\underline{c}} \in \Lambda_{n-2}$ is such that $d_{\mu_{\underline{c}}}$ is obtained from $d_\lam$ by removing $c_1$ and then $c_2$. The parity $\tilde{z}(\lam, \underline{c})$ is computed using Corollary \ref{cor:DS_caps}:
$$ \tilde{z}(\lam, \underline{c}) = z(\lam, c_1) + z(\lam_{c_1}, c_2)$$
where the notation is as in \eqref{eq:sdim_aux1}.

Let $\underline{c}=(c_1, c_2)$ be a pair of caps as above, and let $\mu :=\mu_{\underline{c}}$. Then $\widehat{C}(\mu) = \widehat{C}(\lam) \setminus \{c_1, c_2\}$ with the induced partial order. 

We begin our study of the sum \eqref{eq:sdim_aux2} above with the following observation:

Assume both $c_1, c_2$ are both successors of $c_*$. Then both $(c_1, c_2)$ and $(c_2, c_1)$ are ordered pairs appearing as indices in the sum \eqref{eq:sdim_aux2}, and $\mu_{(c_1, c_2)}=\mu_{(c_2, c_1)}$. By Remark \ref{rmk:z_interpretation}, we have: $$\tilde{z}(\lam, (c_1, c_2)) \equiv \tilde{z}(\lam, (c_2, c_1)) +1 \mod 2.$$

Hence the corresponding terms in the sum \eqref{eq:sdim_aux2} cancel out, and from now on we will consider the sum \eqref{eq:sdim_aux2} so that the sum goes over the ordered pairs $(c_1, c_2)$ where $c_2$ is a successor of $c_1$.

Let us consider the case when $\lam$ is not worthy.  

Let $c' \in \widehat{C}(\lam)$ be a cap (perhaps virtual) with at least $2$ odd successors. 

\InnaC{
\begin{sublemma}
The weight $\mu = \mu_{\underline{c}} \in \Lambda_{n-2}$ is not worthy as well. 
\end{sublemma}
\begin{proof}
Assume the contrary: $\mu$ is worthy.

Recall that since $n\equiv 0 \mod 2$, the virtual cap $c_* \in \widehat{C}(\lam)$ is odd, hence it has an even number of odd successors, by Lemma \ref{lem:odd_even_caps}. After the removal of $c_1, c_2$ it inherits their odd successors, so we have:
$$\{\substack{\text{odd successors of}\\ c_* \text{ in } \widehat{C}(\mu)}\} =   \{\substack{\text{odd successors of}\\ c_* \text{ in } \widehat{C}(\lam)}\} \setminus \{c_1\} \cup \{\substack{\text{odd successors of}\\ c_1 \text{ in } \widehat{C}(\lam)}\} \setminus \{c_2\} \cup \{\substack{\text{odd successors of}\\ c_2 \text{ in } \widehat{C}(\lam)}\}.$$

This immediately implies: since $c_* \in \widehat{C}(\mu)$ has at most one odd successor, the following must hold in $\widehat{C}(\lam)$: $c_* \in \widehat{C}(\lam)$ has no odd successors, $c_1$ is even and has precisely one odd successor: $c_2$, which has no odd successors itself. 

Hence we must have $c' \neq c_*, c_1, c_2$. In this case $c'\in \widehat{C}(\mu)$ will have at least $2$ odd successors, and $\mu$ is not worthy, contradicting our assumption. This proves the statement of the sublemma.
\end{proof}
}

 Applying the induction assumption to each $\mu_{\underline{c}}$, we conclude that if $\lam$ is not worthy, then $$\sdim L_n(\lam) = \sdim DS(L_n(\lam)) = 0.$$
 
 Now let us consider the case when $\lam$ is worthy. Then $c_*$ is odd, and all the maximal (non-virtual) caps in $C(\lam)$ are even. Hence $c_1$ is necessarily even.
 
 Assume $c_2$ is even. Then both $c_1$ and $c_2$ have odd successors, and after the removal of these caps both odd successors will be ``inherited'' by $c_* \in \widehat{C}(\mu)$. Hence $c_*\in \widehat{C}(\mu)$ will have at least $2$ odd successors in $\widehat{C}(\mu)$, and $\mu$ is not worthy. 
 
 Applying the induction assumption to $\mu$, we conclude: if $\lam$ is worthy, the sum in \eqref{eq:sdim_aux2} becomes
 \begin{equation}\label{eq:sdim_aux3}
\sdim L_n(\lam) = \sdim DS(L_n(\lam)) = \sum_{\underline{c}=(c_1, c_2), \; c_1, c_2 \in C(\lam)} (-1)^{\tilde{z}(\lam, \underline{c})} \sdim L_{n-2}(\mu_{\underline{c}})  
 \end{equation}
 over ordered pairs $\underline{c}=(c_1, c_2)$ where $c_1$ is a maximal (non-virtual, even) cap in $C(\lam)$ and $c_2$ is its unique odd successor. 
 
 In that case, the rooted forest $F_{\mu_{\underline{c}}}$ is obtained from $F_{\lam}$ by removing exactly one node, corresponding to the pair $\underline{c}=(c_1, c_2)$. 
 
 The parity $\tilde{z}(\lam, \overline{c})$ is then necessarily $0$: indeed, there is an even number of caps whose right end is to the right of the cap $c_1$, and after its removal, the same is true for the cap $c_2$. By Remark \ref{rmk:z_interpretation}, this implies: $$ \tilde{z}(\lam, \underline{c}) = 0 + 0 =0.$$

 Applying the induction assumption to all $\mu_{\underline{c}}$ and using Corollary \ref{cor:forest_binom}, we obtain: 
  \begin{align*}
\sdim L_n(\lam) &= \sdim DS(L_n(\lam)) =\sum_{\underline{c}=(c_1, c_2), \; c_1, c_2 \in C(\lam)}
L_{n-2}(\mu_{\underline{c}}) = \sum_{\underline{c}=(c_1, c_2), \; c_1, c_2 \in C(\lam)} \frac{\abs{F_{\mu_{\underline{c}}}}!}{F_{\mu_{\underline{c}}}!} = \\
&= \sum_{v \text{ a root of } F_{\lam}} \frac{\abs{F_{\lam}\setminus\{v\}}!}{\left(F_{\lam} \setminus\{v\}\right)!} = \frac{\abs{F_{\lam}}!}{F_{\lam}!}
  \end{align*}
as required. This completes the proof of Theorem \ref{thrm:sdim}.
\end{proof}

As s special case of the statement of Theorem \ref{thrm:sdim}, we have:
\begin{proposition}\label{prop:KW_simples}
 Let $L\in \F_n^k$ be a simple module, and $k \neq 0, \pm 1$. Then $\sdim L =0$.
\end{proposition}
\begin{proof}

Recall from Theorem \ref{thrm:sdim} that $$\sdim L_n(\lam) \neq 0  \; \Longleftrightarrow \; \lam \, \text{  is worthy}$$ So let $\lam \in \Lambda_n$ be a worthy weight. We will show that $L_n(\lam) \in \F_n^k$ with $k =0$ if $n$ is even and $k=\pm 1$ otherwise. In other words, we will prove that 
\begin{equation}\label{eq:KW_aux}                                                                                                                                                                                                                                                                      
     \sum_{i=1}^n (-1)^{\bar{\lam}_i} =\begin{cases}                                                                                                                                                                                                                   
         0 &\text{  if  } n\equiv 0 \mod 2\\
         \pm 1 &\text{  if  } n\equiv 1 \mod 2\\
         \end{cases}.                                                                                                                                                                                                                                                                                                 \end{equation}
 where $\{\bar{\lam}_i\}_{i=1}^n$ are precisely the right ends of the caps in the cap diagram for $\lam$.

Let us prove this by complete induction on $n\geq 1$. 

{\bf Base case:} For $n=1$, the category $\F_1$ only has two blocks: $\F_1^{\pm 1}$, so there is nothing to prove. For $n=2$, the category $\F_2$ has three blocks: $\F_2^0, \F_2^{\pm 2}$. The worthy weights in this case have the form $\lam \in \Lambda_2$ where $\lam_1=\lam_2$, hence $ \sum_{i=1}^2 (-1)^{\bar{\lam}_i}=0$ as required.

{\bf Step}: Let $n\geq 3$, and assume the statement holds up to rank $n-1$. Let $\lam \in \Lambda_n$ be a worthy weight. 

If $n$ is even, the cap diagram for $\lam$ has at least one maximal even cap $c$. Let $c'$ be its unique odd successor. Let $j, j'$ be the indices of the right ends of $c, c'$ respectively. Then $j \neq j' \mod 2$, hence $ (-1)^j + (-1)^{j'}=0$. If we remove both caps $c, c'$, we are left with a cap diagram for a worthy weight in $\Lambda_{n-2}$. By the induction assumption, the statement of \eqref{eq:KW_aux} holds for this weight, so 
$$      \sum_{i: \, \bar{\lam}_i \neq j, j'} (-1)^{\bar{\lam}_i} =0 \;  \Longrightarrow \; \sum_{i=1}^n  (-1)^{\bar{\lam}_i} = 0$$ as required.

If $n$ is odd, the cap diagram for $\lam$ precisely one maximal odd cap $c$. Let $j$ be the index of its right end. If we remove this cap, we are left with a cap diagram for a worthy weight in $\Lambda_{n-1}$. By the induction assumption, the statement of \eqref{eq:KW_aux} holds for this weight, so 
$$      \sum_{i: \, \bar{\lam}_i \neq j} (-1)^{\bar{\lam}_i} =0 \;  \Longrightarrow \; \sum_{i=1}^n (-1)^{\bar{\lam}_i} = \pm 1.$$ This completes the proof of the proposition.

\end{proof}

Finally, we recover the Kac-Wakimoto conjecture for $\p(n)$ proved in \cite{ES_KW}:  
\begin{corollary}\label{cor:KW}
 Let $M \in \F_n^k$ where $k \neq 0, \pm 1$. Then $\sdim M=0$.
\end{corollary}

\end{document}